\documentclass[12pt]{amsart}
\usepackage{leftidx}
\usepackage{mathrsfs}
\usepackage{amssymb}
\usepackage{mathtools}
\usepackage{verbatim}

\newcommand{\tr}{\mathrm{tr}}
\newcommand{\sgn}{\mathrm{sgn}}
\newcommand{\ord}{\mathrm{ord}}
\newcommand{\Mp}{\mathrm{Mp}}
\newcommand{\PGL}{\mathrm{PGL}}
\newcommand{\SL}{\mathrm{SL}}
\newcommand{\Vol}{\mathrm{Vol}}
\newcommand{\Wald}{\mathrm{Wald}}

\newtheorem{definition}{Definition}
\newtheorem{theorem}{Theorem}
\newtheorem{lemma}{Lemma}
\newtheorem{proposition}{Proposition}
\newtheorem{corollary}{Corollary}
\newtheorem*{remark}{Remark}

\numberwithin{equation}{section}
\numberwithin{theorem}{section}
\numberwithin{definition}{section}
\numberwithin{lemma}{section}
\numberwithin{proposition}{section}
\numberwithin{corollary}{section}

\allowdisplaybreaks

\begin{document}
\title{Newforms in the Kohnen plus space}
\author{Ren-He Su}
\address{Graduate school of mathematics, Kyoto University, Kitashirakawa, Kyoto, 606-8502, Japan}
\email{ru-su@math.kyoto-u.ac.jp}
\maketitle

\begin{abstract}
In this paper we want to define the Kohnen plus space for Hilbert modular forms with a odd square-free level and a quadratic character by a representation-theoretic way.
We will show that in the classical case the one we defined is the same with the one given by Kohnen.
Also, we will interpret the actions of the Hecke operators on the new forms in the plus space and give a characterization for the new forms using Hecke operators.
All the results with respect to Hecke operators can are comparable with the integral weight case.
\end{abstract}

\section{Introduction}\label{Introduction}
In 1980, Kohnen \cite{Kohnen:80} introduced a special subspace of classical modular forms with half-integral weight in which the forms are characterized by the Fourier coefficients.
Precisely, if $\sum a(n)\mathbf{e}(nz)$ is such a form with weight $k+1/2$, then $a(n)=0$ unless $(-1)^kn\equiv0$ or $1$ mod $4$.
If we only consider the cusp forms, then the corresponding subspace, denoted by $S^+_{k+1/2}(\Gamma_0(4)),$ is actually an eigenspace of some Hecke operator with respect to some eigenvalue.
An important result is that there exists an isomorphism between $S^+_{k+1/2}(\Gamma_0(4))$ and $S_{2k}(\Gamma_0(1)),$ the space of cusp forms with weight $2k,$ as Hecke modules.

\par
Later, in 1982, Kohnen \cite{Kohnen:82} generalized the plus space to the modular forms with level $4N$ and a quadratic character $\chi$ mod $N,$ where $N$ is a square-free odd integer.
The restriction for the Fourier coefficients of a form in the plus space is now $a(n)=0$ unless $\chi(-1)(-1)^kn\equiv0$ or $1$ mod $4$.
Denote the space by $S^+_{k+1/2}(\Gamma_0(4N),\chi)$.
Kohnen showed that such a plus space is isomorphic to the plus space with trivial character, to which many questions can be reduced.
Let the space of new forms in $S^+_{k+1/2}(\Gamma_0(4N),\chi),$ which is denoted by $S_{k+1/2}^{+,\mathrm{NEW}}(\Gamma_0(4N),\chi),$ be the orthogonal complement of the subspace
\[
\sum_{d\mid N, d<N}\left(S^+_{k+1/2}(\Gamma_0(4d),\chi)+S^+_{k+1/2}(\Gamma_0(4d),\chi)\mid U(N^2/d^2)\right)
\]
with respect to the Petersson product.
Here $U(r^2)$ is the operator replacing the $n$-th Fourier coefficient of a modular form by its $r^2n$-th one.
Kohnen also showed that there exists a linear combination of the Shimura liftings which is a Hecke isomorphism from $S_{k+1/2}^{+,\mathrm{NEW}}(\Gamma_0(4N),\chi)$ onto $S_{2k}^{\mathrm{NEW}}(\Gamma_0(N))$.

\par
In this paper, we want to consider the Hilbert case.

\par
Let $F$ be a totally real number field with ring of integers $\mathfrak{o}$ and the different $\mathfrak{d}_1$.
We fix a square-free odd integral ideal $\mathfrak{I}$ in $\mathfrak{o}$ and a primitive quadratic character $\chi$ mod $\mathfrak{I}$ with conductor $(\mathfrak{f})$ which is a principal ideal generated by some $\mathfrak{f}\in\mathfrak{o}$ such that $\mathfrak{f}\mid\mathfrak{I}$.
For any ideals $\mathfrak{b}$ and $\mathfrak{c}$ such that $\mathfrak{bc}\subset\mathfrak{o}$ we put
\[
\Gamma[\mathfrak{b},\mathfrak{c}]=\left\{
\begin{pmatrix}
a&b\\c&d
\end{pmatrix}\in\SL_2(F)\,\mid\,a,d\in\mathfrak{o}, b\in\mathfrak{b}, c\in\mathfrak{c}
\right\}
\]
and
\[
\Gamma_0(\mathfrak{a})=\Gamma[\mathfrak{d}_1^{-1},\mathfrak{a}\mathfrak{d}_1]
\]
for integral ideal $\mathfrak{a}\subset\mathfrak{o}$
We aim to define the plus space $S^+_{k+1/2}(\Gamma_0(4\mathfrak{I}),\chi)$ using a representation-theoretic method.
For simplicity, here we let $k$ be parallel and take $\mathfrak{f}$ so that it has sign $(-1)^k$.
The non-trivial additive character $\psi_1=\prod_{v\leq\infty}\psi_{1,v}$ on $\mathbb{A}/F$ is the unique one such that $\psi_{1,\infty}(x)$ is equal to $\exp(2\pi\sqrt{-1}x)$ for any infinite place $\infty$ of $F$ and put $\psi(x)=\psi_1(\mathfrak{f}x)$.

\par
Fix a finite place $v$ of $F$ and set $\varpi_v\in F_v$ to be the uniformizer corresponding to $v$.
An element in the metaplectic double covering $\Mp_2(F_v)$ of $\SL_2(F_v)$ has the form $[g,\zeta]$ where $g\in\SL_2(F_v)$ and $\zeta\in\{\pm1\}$.
We let $\mathfrak{d}=\mathfrak{f}\mathfrak{d}_1,$
\[
K_v=
\begin{cases}
\Gamma_0(1)_v&\mbox{ if }v\nmid\mathfrak{I},\\
\Gamma_0(\varpi_v)_v&\mbox{ if }v\mid\mathfrak{f}^{-1}\mathfrak{I},\\
\Gamma[\varpi_v\mathfrak{d}_v^{-1},\mathfrak{d}_v]&\mbox{ if }v\mid\mathfrak{f},
\end{cases}
\]
and
\[
\Gamma_v=
\begin{cases}
\Gamma_0(4)_v&\mbox{ if }v\mid2,\\
K_v&\mbox{ otherwise.}
\end{cases}
\]
Their inverse images in $\Mp_2(F_v)$ are denoted by $\widetilde{K_v}$ and $\widetilde{\Gamma_v},$ respectively
Let $\omega_{\psi_v}$ be the Weil representation of $\Mp_2(F_v)$ on the Schwartz space $\mathcal{S}(F_v)$ with respect to the local additive character $\psi_v$.
Then there exists a genuine character $\varepsilon_v$ of $\widetilde{\Gamma_v}$ given by
\[
\omega_{\psi_v}(\gamma)\phi_{0,v}=\varepsilon_v(\gamma)^{-1}\phi_{0,v}
\]
where $\phi_{0,v}\in\mathcal{S}(F_v)$ is the characteristic function of $\mathfrak{o}_v$.
Let $\widetilde{\mathcal{H}_v}=\widetilde{\mathcal{H}}(\widetilde{\Gamma_v}\backslash\Mp_2(F_v)/\widetilde{\Gamma_v};\varepsilon)$ be the Hecke algebra consisting of compactly supported functions $\varphi$ on $\Mp_2(F_v)$ such that $\varphi(\gamma_1g\gamma_2)=\varepsilon(\gamma_1\gamma_2)\varphi(g)$ for any $\gamma_1, \gamma_2\in\widetilde{\Gamma_v}$.
Now we set
\[
\widetilde{K'_v}=
\begin{cases}
\widetilde{\Gamma[4^{-1}\mathfrak{d}_v^{-1},4\mathfrak{d}_v]}&\mbox{ if }v\mid2,\\
\widetilde{K_v}&\mbox{ otherwise,}
\end{cases}
\]
and
\[
E^K_v(g)=
\begin{cases}
|2|_v^{-1}\Vol(\widetilde{K}_v)^{-1}(\phi_{0,v},\omega_{\psi_v}(g)\phi_{0,v})_v&\mbox{ if }g\in\widetilde{K'_v},\\
0&\mbox{ otherwise,}
\end{cases}
\]
where $(\cdot,\cdot)_v$ denotes the inner product for $\mathcal{S}(F_v)$ and $|\cdot|_v$ and $\Vol$ are normalized so that $|\varpi_v|_v=q_v^{-1}$ and $\Vol(\widetilde{\Gamma_v})=1,$ respectively.
The function $E^K_v$ is a Hecke operator contained in $\widetilde{\mathcal{H}_v}$ and is an idempotent with respect to the involution.
We put $E^K=\prod_{v<\infty}E^K_v\in\widetilde{\mathcal{H}}=\prod_{v<\infty}\widetilde{\mathcal{H}_v}$.

\par

Let $\Mp_2(\mathbb{A})$ be the metaplectic double covering of $\SL_2(\mathbb{A})$ where $\mathbb{A}$ is the adele ring of $F$.
The space $\mathcal{A}^{\tiny\mbox{CUSP}}_{k+1/2}(\SL_2(F)\backslash \Mp_2(\mathbb{A})；\widetilde{\Gamma_{\mathrm{f}}};\varepsilon)$ consists of all cuspidal automorphic forms $\Phi$ on $\Mp_2(\mathbb{A})$ with weight $k+1/2$ which satisfies
\[
\Phi(g\gamma)=\varepsilon(\gamma)^{-1}\Phi(g)\quad(\gamma\in\widetilde{\Gamma_{\mathrm{f}}})
\]
where $\widetilde{\Gamma_{\mathrm{f}}}=\otimes'_{v<\infty}\widetilde{\Gamma_v}$ and $\varepsilon=\prod_{v<\infty}\varepsilon_v$.
Each automorphic form in that space corresponds to exactly one Hilbert cusp form with respect to the congruence subgroup $\Gamma_0(4\mathfrak{I})$ of $\SL_2(F_v)$ with weight $k+1/2$.
The space consisting of all Hilbert cusp forms corresponding to the forms in $\mathcal{A}^{\tiny\mbox{CUSP}}_{k+1/2}(\SL_2(F)\backslash \Mp_2(\mathbb{A})；\widetilde{\Gamma_{\mathrm{f}}};\varepsilon)$ is denoted by $S_{k+1/2}(\Gamma_0(4\mathfrak{I}),\chi_\mathfrak{f})$.
Since the global Hecke algebra $\widetilde{\mathcal{H}}$ acts on $\mathcal{A}^{\tiny\mbox{CUSP}}_{k+1/2}(\SL_2(F)\backslash \Mp_2(\mathbb{A})；\widetilde{\Gamma_{\mathrm{f}}};\varepsilon)$ by the right translation $\rho,$ it also acts on $S_{k+1/2}(\Gamma_0(4\mathfrak{I}),\chi_\mathfrak{f})$ equivalently.
We call the $E^K$-fixed subspace $S_{k+1/2}(\Gamma_0(4\mathfrak{I}),\chi_\mathfrak{f}),$ which is denoted by $S^+_{k+1/2}(\Gamma_0(4\mathfrak{I}),\chi_\mathfrak{f}),$ the Kohnen plus space.

\par

Similar to the case $F=\mathbb{Q},$ the plus space can be characterized by some properties of the Fourier coefficients of the forms in it.
\begin{theorem}
	The Kohnen plus space $S^+_{k+1/2}(\Gamma_0(4\mathfrak{I}),\chi_\mathfrak{f})$ is the subspace of $S_{k+1/2}(4\Gamma_0(\mathfrak{I}),\chi_\mathfrak{f})$ which consists of the forms whose $\xi$-th Fourier coefficient vanishes unless there exists some $\lambda\in\mathfrak{o}$ such that $\xi-\mathfrak{f}\lambda^2\in4\mathfrak{o}$.
\end{theorem}

The paper will focus on the new forms in the plus space.
For a finite place $v$ of $F$, any Hecke operator in $\widetilde{\mathcal{H}_v}$ of the form $E^K_v\ast\varphi$ leaves $S^+_{k+1/2}(\Gamma_0(4\mathfrak{I}),\chi_\mathfrak{f})$ invariant, where $\ast$ is the involution of Hecke operators and $\varphi\in\widetilde{\mathcal{H}_v}$.
Note that for any odd integral ideal $\mathfrak{I}'$ such that $\mathfrak{f}\mid\mathfrak{I}'$ and $\mathfrak{I}'\mid\mathfrak{I},$ the space $S^+_{k+1/2}(\Gamma_0(4\mathfrak{I}'),\chi_\mathfrak{f})$ is a subspace of $S^+_{k+1/2}(\Gamma_0(4\mathfrak{I}),\chi_\mathfrak{f})$.
Now the subspace of $S^+_{k+1/2}(\Gamma_0(\mathfrak{4I}),\chi_\mathfrak{f})$ spanned by any cusp form in the image of some $S^+_{k+1/2}(\Gamma_0(4\mathfrak{I}'),\chi_\mathfrak{f})$ under $E^K_v\ast\widetilde{\mathcal{H}_v}$ for some $v<\infty$ is called the old space, denoted by $S^{+,\mathrm{OLD}}_{k+1/2}(\Gamma_0(\mathfrak{4I}),\chi_\mathfrak{f})$.
The orthogonal complement of the old space with respect to the Petersson product is called the new space, denoted by $S^{+,\mathrm{NEW}}_{k+1/2}(\Gamma_0(\mathfrak{4I}),\chi_\mathfrak{f})$.
Any cusp form in the new space is called a new form.

\par

Given a finite place $v$ of $F,$ a cusp form $f\in S^+_{k+1/2}(\Gamma_0(\mathfrak{4I}),\chi_\mathfrak{f})$ generates a representation $\rho_v$ of $\Mp_2(F_v)$ via the right translation, that is, the space spanned by $\{\rho(g)f\mid g\in\Mp_2(F_v)\}$.
If $\rho_v$ is irreducible for all $v<\infty,$ the cusp form $f$ is called a primitive form.
In that case, each $\rho_v$ is contained in some principal series representation $\tilde{I}_{\psi_v}(s_v)$ where $s_v\in\mathbb{C}$.
It is well-known that we can take a basis $\mathcal{B}$ of $S^{+,\mathrm{NEW}}_{k+1/2}(\Gamma_0(\mathfrak{4I}),\chi_\mathfrak{f})$ consisting of primitive forms, which are unique up to multiplication with non-zero complex numbers.
If $f\in\mathcal{B}$ and $v\mid\mathfrak{I},$ the local representation $\rho_v$ is equivalent to a so-called Steinberg representation, twisted or non-twisted, which is a subrepresentation of $\tilde{I}_{\psi_v}(s_v)$ with some $s_v$ such that $q_v^{2s_v}=q_v$.

\par

For a finite place $v$ of $F$ dividing $\mathfrak{I},$ set $\beta_v=1$ or $-1$ according $v\mid\mathfrak{f}^{-1}\mathfrak{I}$ or $v\mid\mathfrak{f},$ respectively.
we let $\widetilde{\mathcal{T}}_{1,v}\in\widetilde{\mathcal{H}_v}$ and $\widetilde{\mathcal{U}}_{1,v}\in\widetilde{\mathcal{H}_v}$ be the Hecke operators which are supported on $\widetilde{\Gamma_v}\left[\begin{pmatrix}
\varpi_v^{\beta_v}&0\\0&\varpi_v^{-\beta_v}
\end{pmatrix},1\right]\widetilde{\Gamma_v}$ and $\widetilde{\Gamma_v}\left[\begin{pmatrix}
0&-\boldsymbol{\delta}_v^{-1}\varpi_v^{-\beta_v}\\\boldsymbol{\delta}_v\varpi_v^{\beta_v}&0
\end{pmatrix},1\right]\widetilde{\Gamma_v}$ such that
\[
\widetilde{\mathcal{T}}_{1,v}\left(\left[\begin{pmatrix}
\varpi_v^{\beta_v}&0\\0&\varpi_v^{-\beta_v}
\end{pmatrix},1\right]\right)=q_v^{-1/2}\frac{\alpha_{\psi_v}(\varpi_v)}{\alpha_{\psi_v}(1)}
\]
anr
\[
\widetilde{\mathcal{U}}_{1,v}\left(\left[\begin{pmatrix}
0&-\boldsymbol{\delta}_v^{-1}\varpi_v^{-\beta_v}\\\boldsymbol{\delta}_v\varpi_v^{\beta_v}&0
\end{pmatrix},1\right]\right)=\alpha_{\psi_v}(\boldsymbol{\delta}_v\varpi_v),
\]
respectively, where $\alpha_{\psi_v}$ is the Weil index with respect to $\psi_v$ and $\boldsymbol{\delta}_v$ satisfies $\boldsymbol{\delta}_v\mathfrak{o}_v=\mathfrak{d}_v$.
Using these Hecke operators, we give a necessary and sufficient condition to determine if a primitive form is new.
This result is inspired by a recent work from Baruch and Purkait \cite{BaPu:15}.

\begin{theorem}
	Let $f\in S^+_{k+1/2}(\Gamma_0(4\mathfrak{I}),\chi_\mathfrak{f})$ be a primitive form.
	Then $f$ is a new form if and only if
	\[
	\rho_v(\widetilde{\mathcal{T}}_{1,v}\widetilde{\mathcal{U}}_{1,v})f=-f=\rho_v(\widetilde{\mathcal{U}}_{1,v}\widetilde{\mathcal{T}}_{1,v})
	\]
	for any $v$ dividing $\mathfrak{I}$.
\end{theorem}

The Hecke operator $\widetilde{\mathcal{U}}_{1,v}$ is called the Atkin-Lehner involution.
It has eigenvalues $1$ and $-1$ on the plus space.

\par

For any finite place $v$ not dividing $\mathfrak{f},$ we may defined $\widetilde{\mathcal{T}}_{1,v}$ in the same manner with the ones for $v\mid\mathfrak{f}^{-1}\mathfrak{I}$.
The eigenvalue of $E^K_v\ast\widetilde{\mathcal{T}}_{1,v}$ with respect to a new primitive form $f$ can be calculated explicitly in a representation-theoretic way.

\begin{proposition}
	Fix a finite place $v$ of $F$.
	Let $f$ be a new primitive form such that the local representation of $\Mp_2(F_v)$ generated by $f$ is contained in the principal series $\tilde{I}_{\psi_v}(s_v)$ for some $s_v\in\mathbb{C}$.
	Then we have
	\[
	\rho_v(E^K_v\ast\widetilde{\mathcal{T}}_{1,v})f=\lambda_vf
	\]
	where
	\[
	\lambda_v=\begin{cases}
	q_v^{1/2}(q_v^{s_v}+q_v^{-s_v})&\mbox{ if }v\nmid2\mathfrak{I},\\
	(1+q_v^{-1})^{-1}q_v^{1/2}(q_v^{s_v}+q_v^{-s_v})&\mbox{ if }v\mid2,\\
	q_v^{-1/2+s_v}&\mbox{ if }v\mid\mathfrak{I}.
	\end{cases}
	\]
\end{proposition}

Finally, let $K_0(\mathfrak{I}_v)\subset \PGL_2(F_v)$ be the maximal compact subgroup of $\PGL_2(F_v)$ in which the matrices have the lower-left entry in $\mathfrak{d}_v$ if $\infty>v\nmid\mathfrak{I}$ or be the Iwahori subgroup if $v\mid\mathfrak{I}$.
Also, we put $\mathcal{A}^{\mathrm{CUSP}}_{2k}(\mathfrak{I}_v)=\mathcal{A}^{\mathrm{CUSP}}_{2k}(\PGL_2(F)\backslash\PGL_2(\mathbb{A})/\prod_{v<\infty}K_0(\mathfrak{I}_v))$ to be the space of Automorphic forms on $\PGL_2(F)\backslash\PGL_2(\mathbb{A})/\prod_{v<\infty}K_0(\mathfrak{I}_v)$.
We will use Waldspurger's theory on Shimura correspondence to show the following theorem.

\begin{theorem}
	There exists a Hecke isomorphism between $S_{k+1/2}^{+,\mathrm{NEW}}(\Gamma_0(4N),\chi_\mathfrak{f})$ and the space $\mathcal{A}^{\mathrm{NEW,CUSP}}_{2k}(\mathfrak{I}_v)$
	which is spanned by the Hecke forms in $\mathcal{A}^{\mathrm{CUSP}}_{2k}(\mathfrak{I}_v)$ which generate a Steinberg representation at any place dividing $\mathfrak{I}$.
\end{theorem}

%%%%%%%%%%%%%%%%%%%%%%%%%%%%%%%%%%%%%%%%%%%%%%%%%%%%%%%%%%%%%%%%%%%%%%%%%%%%%%%%%%%%%%%%%%%%%%%%%%%%%%%%%%%%%%%%%%%%%%%%%%%%%%%%%%%%%%%%%%%%%%%%%%%%%%%%%%%%%%%%%%%%%%%%%%%%%%%%%%%%%%%%%%%%%%%%%%%%%%%%%%%%%%%%%%%%%%%%%%%%%%%%%%%%%%%%%%%%%%%%%%%%%%%%%%%%%%%%%%%%%%%%%%%%%%%%%%%%%%%%%%%%%%%%%%%%%%%%%%%%%%%%%%%%%%%%%%%%%%%%%%%%%%%%%%%%%%%%%%%%%%%%%%%%%%%%%%%%%%%%%%%%%%%%%%%%%%%%%%%%%%%%%%%%%%%%%%%%%%%%%%%%%%%%%%%%%%%%%%%%%%%%%%%%%%%%%%%%%%%%%%%%%%%%%%%%%%%%%%%%%%%%%%%%%%%%%%%

\section{Weil representation}\label{Weil representation}
Let $F$ be a local field with characteristic $0$.
If $F$ is archimedean, we assume $F=\mathbb{R}$.
If $F$ is a finite extension of $\mathbb{Q}_p$ for some prime $p,$ let $\mathfrak{o}$ and $\mathfrak{p}$ denote its ring of integers and maximal ideal, respectively.
The order of the residue field $\mathfrak{o}/\mathfrak{p}$ is denoted by $q$ and fix a uniformizer $\varpi$.

\par
Let $\psi:F\rightarrow\mathbb{C}^\times$ be a non-trivial additive unitary character.
If $F=\mathbb{R},$ we set $\psi(x)=\mathbf{e}(x)$ or $\mathbf{e}(-x)$.
If $F$ is non-archimedean, we denote the index of $\psi,$ the largest integer $c$ such that $\psi(\mathfrak{p}^{-c})=1,$ by $c_\psi$.
We fix an element $\boldsymbol\delta$ of order $c_\psi$ and let $\mathfrak{d}=\mathfrak{p}^{c_\psi}$.
If $F=\mathbb{R},$ we let $\boldsymbol\delta=1$.

\par
The metaplectic double covering of $\SL_2(F)$ is
\[
\Mp_2(F)=\left\{[g,\zeta]\,\mid\,g\in \SL_2(F), \zeta\in{\pm1}\right\}
\]
where the binary operation is given by
\[
[g_1,\zeta_1][g_2,\zeta_2]=[g_1g_2,\zeta_1\zeta_2\boldsymbol{c}(g_1,g_2)].
\]
The explicit formula for the 2-cocycle $\boldsymbol{c}$ is
\[
\boldsymbol{c}(g_1,g_2)=(\frac{\boldsymbol b(g_1)}{\boldsymbol b(g_1g_1)},\frac{\boldsymbol b(g_2)}{\boldsymbol b(g_1g_1)})
\]
where $(\quad,\quad)$ is the quadratic Hilbert symbol and
\[
\boldsymbol{b}\left(\begin{pmatrix}a&b\\c&d\end{pmatrix}\right)
=\begin{cases}
c\quad&\mbox{ if }c\neq0,\\
d\quad&\mbox{ otherwise. }
\end{cases}
\]
For simplicity, we set $[g]=[g,1]\in \Mp_2(F)$ for $g\in \SL_2(F)$ and
\begin{align*}
	\mathbf{u}^\sharp(b)&=\left[\begin{pmatrix}1&b\\0&1\end{pmatrix}\right],&\mathbf{u}^\flat(b)=\left[\begin{pmatrix}1&0\\b&1\end{pmatrix}\right],\quad&\mbox{for } b\in F,\\
	\mathbf{m}(a)&=\left[\begin{pmatrix}a&0\\0&a^{-1}\end{pmatrix}\right],&\mathbf{w}_a=\left[\begin{pmatrix}0&-a^{-1}\\a&0\end{pmatrix}\right],\quad&\mbox{for } a\in F^\times.
\end{align*}
Also, in the calculations among $\Mp_2(F)$, we put $[1,\zeta]=\zeta$ for $\zeta\in\{\pm1\}$ if there is no confusion.
For any subset $S$ of $\SL_2(F),$ we let $\widetilde{S}$ denote the inverse image of $S$ in $\Mp_2(F)$.

\par
Let $dy$ be the usual Lebesgue measure on $F$ if $F=\mathbb{R}$ or the normalized Haar measure such that $\mathfrak{o}$ has volumn $1$ if $F$ is non-archimedean.
For any function $\phi$ in the Schwartz space $\mathbb{S}(F)$ of $F,$ its Fourier transform is defined by
\[
\hat{\phi}(x)=|\boldsymbol\delta|^{1/2}\int_{F}\phi(y)\psi(xy)dy.
\]
Note that $|\boldsymbol\delta|^{1/2}dy$ forms the self-dual Haar measure for this Fourier transform.

\par
For any $a\in F^{\times},$ it is known that, regardless of the choice of $\phi\in\mathbb{S}(F),$ there exists a constant $\alpha_\psi(a)$ depending on $a\in F^\times/(F^\times)^2$ satisfying
\begin{equation}\label{defofWeilconst}
\int_{F}\phi(y)\psi(ay^2)dy=\alpha_\psi(a)|2a|^{-1/2}\int_{F}\hat{\phi}(y)\psi\left(-\frac{y^2}{4a}\right)dy.
\end{equation}
It is called the Weil constant or Weil index.
It is an eighth root of $1$ and satisfies $\alpha_{\psi_b}(a)=\alpha_\psi(ba)$ for any $b\in F^\times$ where $\psi_b(x)=\psi(bx)$.
We also have $\alpha_\psi(-a)=\overline{\alpha_\psi(a)}$ and
\begin{equation}
\frac{\alpha_\psi(1)\alpha_\psi(a_1a_2)}{\alpha_\psi(a_1)\alpha_\psi(a_2)}=(a_1,a_2).
\end{equation}
Furthermore, by taking $\phi$ to be the characteristic function of $\mathfrak{o},$ one can see that $\alpha_\psi(a)=1$ for $a\in\boldsymbol{\delta}\mathfrak{o}^\times$ if $q$ is odd.
\begin{lemma}\label{lem:sumofWeilconsts}
If $q$ is odd, We have 
\[
\sum_{u\in\mathfrak{o}^\times/(1+\mathfrak{p})}\alpha_\psi(\boldsymbol{\delta}\varpi u)=0.
\]
\end{lemma}
\begin{proof}
See the proof of Lemma 1.1 in \cite{Ikeda:17}.
\end{proof}

\par
Now let us define the Weil representation $\omega_\psi$.
It is a representation of $\Mp_2(F)$ on $\mathbb{S}(F)$ given by
\begin{align*}
\omega_\psi(\mathbf{u}^\sharp(b))\phi(x)&=\psi(bx^2)\phi(x),\\
\omega_\psi(\mathbf{m}(a))\phi(x)&=\frac{\alpha_\psi(1)}{\alpha_\psi(a)}|a|^{1/2}\phi(ax),\\
\omega_\psi(\mathbf{w}_a)\phi(x)&=\overline{\alpha_\psi(a)}|2a^{-1}|^{1/2}\hat{\phi}(-2a^{-1}x).
\end{align*}
From the definition, one can easily check if $b\neq0,$ then
\[
\omega_\psi(\mathbf{u}^\flat(b))\phi(x)=\overline{\alpha_\psi(b)}|2\boldsymbol\delta b^{-1}|^{1/2}\int_{F}\phi(t+y)\phi(b^{-1}y^2)dy.
\]
The Weil representation $\omega_\psi$ is unitary with respect to the inner product
\[
(\phi_1,\phi_2)=\int_F\phi_1(y)\overline{\phi_2(y)}dy.
\]

\par
Now suppose that $F$ is non-archimedean.
If $\mathfrak{b}$ and $\mathfrak{c}$ are two rational ideals in $F$ such that $\mathfrak{bc}\subset\mathfrak{o},$ we define the congruence subgroup
\[
\Gamma[\mathfrak{b},\mathfrak{c}]=\left\{\begin{pmatrix}a&b\\c&d\end{pmatrix}\in \SL_2(F)\,\bigg|\,a,d\in\mathfrak{o}, b\in\mathfrak{b}, c\in\mathfrak{c} 
\right\}
\]
and let
\[
\Gamma_0(\xi)=\Gamma[\mathfrak{d}^{-1},\xi\mathfrak{d}]
\]
for any $\xi\in\mathfrak{o}/\{0\}$.
The following lemma can be proved by a direct calculation.
\begin{lemma}\label{defofvarepsilon}
Let $\phi_0=\mathbf{1}_\mathfrak{o}\in\mathbb{S}(F)$ be the characteristic function of $\mathfrak{o}$.
Then there exists a geniune character $\varepsilon:\widetilde{\Gamma_0(4)}\rightarrow\mathbb{C}^\times$ given by
\[
\omega_\psi(\gamma)\phi_0=\varepsilon(\gamma)^{-1}\phi_0.
\]
\end{lemma}
One can easily check that
\begin{align*}
\omega_\psi(\mathbf{m}(a))\phi_0&=\frac{\alpha_\psi(1)}{\alpha_\psi(a)}\phi_0\quad(a\in\mathfrak{o}^\times),\\
\omega_\psi(\mathbf{u}^\sharp(b))\phi_0&=\phi_0\quad(b\in\mathfrak{d}^{-1}),\\
\omega_\psi(\mathbf{u}^\flat(c))\phi_0&=\phi_0\quad(c\in4\mathfrak{d}),
\end{align*}
and if $q$ is odd,
\[
\omega_\psi(\mathbf{w}_{\boldsymbol{\delta}})\phi=\phi_0.
\]
\par
Thus we can write down the formula for the character explicitly.
If $q$ is odd, then $\Gamma_0(4)$ is perfect so there is a unique splitting $\mathbf{s}:\Gamma_0(4)\rightarrow\widetilde{\Gamma_0(4)}$.
Now $\varepsilon$ is given by $\varepsilon([g,\zeta])=\mathbf{s}(g)\zeta$.
In fact, we have
\[
\varepsilon\left(\left[\begin{pmatrix}a&b\\c&d\end{pmatrix},\zeta\right]\right)=
\begin{dcases*}
\frac{\alpha_\psi(d)}{\alpha_\psi(1)}\zeta&\mbox{ if }$c=0$,\\
\zeta&\mbox{ if }$c\neq0\mbox{ and }d\in\mathfrak{p}$,\\
\frac{\alpha_\psi(c)}{\alpha_\psi(cd)}\zeta&\mbox{ otherwsie,}
\end{dcases*}
\]
where for checking of the second and third cases, one can use the decompositions
\[
\left[\begin{pmatrix}
a&b\\c&d
\end{pmatrix},\zeta\right]=\mathbf{m}(-\boldsymbol{\delta}b)\mathbf{w}_{\boldsymbol{\delta}}\mathbf{u}^\sharp(-bd)\mathbf{u}^\flat(a/b)\times\zeta(\boldsymbol{\delta},-\boldsymbol{\delta}b)
\]
and
\[
\left[\begin{pmatrix}
a&b\\c&d
\end{pmatrix},\zeta\right]=\mathbf{u}^\sharp(b/d)\mathbf{m}(d^{-1})\mathbf{u}^\flat(c/d)\times\zeta(c,d),
\]
respectively.
If $q$ is even, the formula for $\varepsilon$ is
\[
\varepsilon\left(\left[\begin{pmatrix}a&b\\c&d\end{pmatrix},\zeta\right]\right)=
\begin{dcases*}
\frac{\alpha_\psi(d)}{\alpha_\psi(1)}\zeta&\mbox{ if }c=0,\\
\frac{\alpha_\psi(c)}{\alpha_\psi(cd)}\zeta&\mbox{ otherwsie.}
\end{dcases*}
\] 
\par
Notice that $\mathbf{m}(2)\widetilde{\Gamma_0(4)}\mathbf{m}(2)^{-1}=\widetilde{\Gamma[4\mathfrak{d}^{-1},\mathfrak{d}]}$.
We can define a genuine character $\check{\varepsilon}$ of $\widetilde{\Gamma[4\mathfrak{d}^{-1},\mathfrak{d}]}$ by
\[
\check{\varepsilon}(\gamma)=\varepsilon(\mathbf{m}(2)^{-1}\gamma\mathbf{m}(2)).
\]
Since 
\[
\omega_\psi(\mathbf{m}(2)^{-1})\phi_0'=\frac{\alpha_\psi(2)}{\alpha_\psi(1)}|2|^{-1/2}\phi_0
\]
where $\phi_0'\in\mathbb{S}(F)$ is the characteristic function of $\mathfrak{p}/2,$ one can easily see that
\[
\omega_\psi(\gamma)\phi_0'=\check{\varepsilon}^{-1}(\gamma)\phi_0'
\]
for $\gamma\in\widetilde{\Gamma[4\mathfrak{d}^{-1},\mathfrak{d}]}$.
We write these results into a lemma.
\begin{lemma}\label{defofcheckvarepsilon}
	Let $\phi_0'=\mathbf{1}_{\mathfrak{o}/2}\in\mathbb{S}(F)$ be the characteristic function of $\mathfrak{o}/2$.
	Then there exists a geniune character $\check{\varepsilon}:\widetilde{\Gamma[4\mathfrak{d}^{-1},\mathfrak{d}]}\rightarrow\mathbb{C}^\times$ given by
	\[
	\omega_\psi(\gamma)\phi_0'=\check{\varepsilon}(\gamma)^{-1}\phi_0'.
	\]
	Moreover, the formula of $\check{\varepsilon}$ is given by
	\[
	\check{\varepsilon}(\gamma)=\varepsilon(\mathbf{m}(2)^{-1}\gamma\mathbf{m}(2)).
	\]
\end{lemma}

%%%%%%%%%%%%%%%%%%%%%%%%%%%%%%%%%%%%%%%%%%%%%%%%%%%%%%%%%%%%%%%%%%%%%%%%%%%%%%%%%%%%%%%%%%%%%%%%%%%%%%%%%%%%%%%%%%%%%%%%%%%%%%%%%%%%%%%%%%%%%%%%%%%%%%%%%%

\section{The Idempotents $e^K$ and $E^K$}\label{The Idempotents $e^K$ and $E^K$}
In this section, we want to define two local idempotent Hecke operators $e^K$ and $E^K$.
They are essential in defining the plus space.
We let $F$ be a non-archimedean field.
Set
\[
K=
\begin{cases}
\Gamma_0(1)&\mbox{ if }q\mbox{ is even,}\\
\Gamma_0(1)\mbox{ ,}\Gamma_0(\varpi)\mbox{ or }\Gamma[\varpi\mathfrak{d}^{-1},\mathfrak{d}]&\mbox{ if }q\mbox{ is odd.}
\end{cases}
\]
Also, we put
\[
\Gamma=
\begin{cases}
\Gamma_0(4)&\mbox{ if }\Gamma=\Gamma_0(1),\\
K&\mbox{ otherwise.}
\end{cases}
\]
If we define the Schwartz space $\mathbb{S}(2^{-1}\mathfrak{o}/\mathfrak{o})$ by
\begin{align*}
&\mathbb{S}(2^{-1}\mathfrak{o}/\mathfrak{o})
\\=&\left\{
\phi\in\mathbb{S}(F)\,|\,\mbox{Supp}(\phi)\subset2^{-1}\mathfrak{o}, \phi(x+y)=\phi(x)\mbox{ for }x\in2^{-1}\mathfrak{o}, y\in\mathfrak{o} \right\},
\end{align*}
then we have the following proposition, which is Prop. 3.3 in \cite{HiraIke:13} by Hiraga and Ikeda.
Note that if $q$ is odd then $\mathbb{S}(2^{-1}\mathfrak{o}/\mathfrak{o})$ is of one dimension.
\begin{proposition}
The Schwartz space $\mathbb{S}(2^{-1}\mathfrak{o}/\mathfrak{o})$ forms an invariant and irreducible subspace of $\mathbb{S}(F)$ under the action of $\widetilde{K}$ via $\omega_\psi$. We denote this irreducible representation by $\Omega_\psi$.
\end{proposition}
Recall that there exists an genuine character $\varepsilon$ of $\widetilde{\Gamma}$ given by $\varepsilon(\gamma)^{-1}\phi_0=\omega_\psi(\gamma)\phi_0$ by Lemma \ref{defofvarepsilon}.
The Hecke algebra $\widetilde{\mathcal{H}}=\widetilde{\mathcal{H}}(\widetilde{\Gamma}\backslash \Mp_2(F)/\widetilde{\Gamma};\varepsilon)$ consits of compactly supported geniune functions $\varphi$ on $Mp_(F)$ which satisfies $\varphi(\gamma_1 g\gamma_2)=\varepsilon(\gamma_1\gamma_2)\varphi(g)$ for $\gamma_1, \gamma_2\in\widetilde{\Gamma}$.
The multiplicative operator in $\widetilde{\mathcal{H}}$ is given by the convolution product
\[
(\varphi_1\ast\varphi_2)(g)=\int_{\Mp_2(F)}\varphi_1(h)\varphi_2(h^{-1}g)dh.
\]
Here $dh$ is the Haar measure on $\Mp_2(F)$ normalized so that the volume of $\widetilde{\Gamma}$ is $1$.
Sometimes we might drop the notation $\ast$ for simplicity.
\par
Now by $e^K$ and $E^K$ we denote two Hecke operators by the following definition.
One can easily check that they are actually in $\widetilde{\mathcal{H}}$.
We let $e$ be the order of $2$ in $F,$ that is, $|2|=q^{-e}$.
\begin{definition}\label{defofeKandEK}
The two Hecke operators $e^K$ and $E^K$ in $\widetilde{\mathcal{H}}$ are given by
\[
e^K(g)=
\begin{cases}
q^e\Vol(\widetilde{K})^{-1}(\phi_0,\omega_\psi(g)\phi_0)&\mbox{ if }g\in\widetilde{K},\\
0&\mbox{ otherwise,}
\end{cases}
\]
and
\[
E^K(g)=
\begin{cases}
e^K(\mathbf{w}_{2\boldsymbol\delta}^{-1}g\mathbf{w}_{2\boldsymbol\delta})&\mbox{ if }q\mbox{ is even,}\\
e^K(g)&\mbox{ otherwise.}
\end{cases}
\]
\end{definition}
For $q$ even, the definition of $E^K$ implies that it is supported on $\mathbf{w}_{2\boldsymbol\delta}\widetilde{K}\mathbf{w}_{2\boldsymbol\delta}^{-1}=\widetilde{\Gamma[4^{-1}\mathfrak{d}^{-1},4\mathfrak{d}]}$.
Since
\[
\omega_\psi(\mathbf{w}_{2\boldsymbol\delta})\phi_0=\overline{\alpha_\psi(2\boldsymbol{\delta})}\phi_0,
\]
by the unitarity of $\omega_\psi,$ we have
\begin{equation}\label{eq:formulaofEK}
E^K(g)=
\begin{cases}
q^e\Vol(\widetilde{K})^{-1}(\phi_0,\omega_\psi(g)\phi_0)&\mbox{ if }g\in\widetilde{\Gamma[4^{-1}\mathfrak{d}^{-1},4\mathfrak{d}]},\\
0&\mbox{ otherwise.}
\end{cases}
\end{equation}
By Schur's orthogonality relations, we have
\begin{proposition}
The Hecke operators $e^K$ and $E^K$ are idempotents, that is, we have
\[
e^K\ast e^K=e^K
\]
and
\[
E^K\ast E^K=E^K.
\]
\end{proposition}
\par
We can get explicit values of $e^K$ and $E^K$.
The explicit values of $e^K$ was calculated by Hiraga and Ikeda in Lemma 3.5 of \cite{HiraIke:13}.
Note that the Haar measure is different from the one used in \cite{HiraIke:13}.
\begin{lemma}\label{lem:valuesofeK}
The following statements are true.\\
\begin{enumerate}
\item For non-zero $z\in\mathfrak{o},$ if $e^K$ is defined on $\mathbf{u}^\flat(\boldsymbol{\delta}z),$ we have
\begin{align*}
e^K(\mathbf{u}^\flat(\boldsymbol{\delta}z))=&\Vol(\widetilde{K})^{-1}\alpha_\psi(\boldsymbol{\delta}z)|2z|^{-1/2}\int_{\mathfrak{o}}\overline{\psi\left(\boldsymbol{\delta}^{-1}y^2/z\right)}dy\\
=&\Vol(\widetilde{K})^{-1}|2|^{-1}\int_{\mathfrak{o}}\psi(\boldsymbol{\delta}^{-1}zy^2/4)dy.
\end{align*}
\item In the case $K=\Gamma_0(1),$ for $g=\begin{pmatrix}
a&\boldsymbol{\delta}^{-1}b\\\boldsymbol{\delta}c&d
\end{pmatrix}\in K$ with $c\in\mathfrak{o}^\times,$ we have
\[
e^K([g])=\Vol(\widetilde{K})^{-1}\alpha_\psi(\boldsymbol{\delta}c)|2|^{-1/2}.
\]
\end{enumerate}
\end{lemma}
Consequently, by the definition, we can get the explicit values of $E^K$.
\begin{lemma}We assume that $q$ is even.
	\begin{enumerate}
		\item For non-zero $z\in\mathfrak{o},$ we have
		\begin{align*}
			E^K(\mathbf{u}^\sharp(4^{-1}\boldsymbol{\delta}^{-1}z))=&\Vol(\widetilde{K})^{-1}\overline{\alpha_\psi(\boldsymbol{\delta}z)}|2z|^{-1/2}\int_{\mathfrak{o}}\psi\left(\boldsymbol{\delta}^{-1}y^2/z\right)dy\\
			=&\Vol(\widetilde{K})^{-1}|2|^{-1}\int_{\mathfrak{o}}\overline{\psi(\boldsymbol{\delta}^{-1}zy^2/4)}dy.
		\end{align*}
		\item For $c\in\mathfrak{o}^\times$ we have
		\[
		E^K(\mathbf{w}_{4\boldsymbol{\delta}c})=\Vol(\widetilde{K})^{-1}\alpha_\psi(\boldsymbol{\delta}c)|2|^{-1/2}.
		\]
	\end{enumerate}
\end{lemma}
\par
Now fix $s\in\mathbb{C}$ and set $B$ to be the subgroup of $\SL_2(F)$ consisting of upper-triangular elements.
We let $\tilde{I}_\psi(s)$ be the principal series representation of $\Mp_2(F)$ induced from the genuine character of $\widetilde{B}$ given by
\[
\mathbf{u}^\sharp(b)\mathbf{m}(a)\mapsto\frac{\alpha_\psi(1)}{\alpha_\psi(a)}|a|^s.
\]
Thus $\tilde{I}_\psi(s)$ consists of all continuous genuine functions $f$ on $\Mp_2(F)$ such that
\[
f(\mathbf{u}^\sharp(b)\mathbf{m}(a)h)=\frac{\alpha_\psi(1)}{\alpha_\psi(a)}|a|^{s+1}f(g)
\]
and $\Mp_2(F)$ acts on it by right translation $\rho$.
Note that any $f\in\tilde{I}_\psi(s)$ is defined by its values on $\widetilde{\Gamma_0(1)}$.
Thus any Hecke operator $\varphi\in\widetilde{\mathcal{H}}$ acts on $\tilde{I}_\psi(s)$ by
\[
(\rho(\varphi)f)(g)=\int_{\Mp_2(F)}\varphi(h)f(gh)dh.
\]
The next proposition given by Hiraga and Ikeda in \cite{HiraIke:13} relates to the nature of the Kohnen plus space.
\begin{proposition}\label{prop:EKfixedsubsp}
Assume that $K=\Gamma_0(1)$.
Let $f^{[0]}\in\tilde{I}_\psi(s)$ be the one whose restriction on $\widetilde{\Gamma_0(1)}$ is equal to $q^{-e}\Vol(\widetilde{K})\cdot\overline{e^K}$ and $f^+=\rho(\mathbf{w}_{2\boldsymbol{\delta}})f^{[0]}$.
The fixed subspaces of $\tilde{I}_\psi(s)$ by $e^K$ and $E^K,$ denoted by $\tilde{I}_\psi(s)^{e^K}$ and $\tilde{I}_\psi(s)^{E^K},$ respectively, are given by
\[
\tilde{I}_\psi(s)^{e^K}=\mathbb{C}\cdot f^{[0]}
\]
and
\[
\tilde{I}_\psi(s)^{E^K}=\mathbb{C}\cdot f^+.
\]
In particular, they are of one dimension.
\end{proposition}

%%%%%%%%%%%%%%%%%%%%%%%%%%%%%%%%%%%%%%%%%%%%%%%%%%%%%%%%%%%%%%%%%%%%%%%%%%%%%%%%%%%%%%%%%%%%%%%%%%%%%%%%%%%%%%%%%%%%%%%%%%%%%%%%%%%%%%%%%%%%%%%%%%%%%%%%%%

\section{The Hecke algebras}\label{The Hecke algebras}

In this section we want to take a look at the structure of the Hecke algebras $\widetilde{\mathcal{H}}(\widetilde{\Gamma}\backslash \Mp_2(F)/\widetilde{\Gamma};\varepsilon)$ defined in the last section.
\par
Now we assume that $q$ is odd and put 
\[
\Gamma=\begin{cases}
\Gamma_1&\mbox{ if }\Gamma=\Gamma_0(1),\\
\Gamma_2&\mbox{ if }\Gamma=\Gamma_0(\varpi)\mbox{ or }\Gamma[\varpi\mathfrak{d}^{-1},\mathfrak{d}].
\end{cases}
\]
Furthermore, for the second case, we set
\[
\boldsymbol{\mu}=\begin{cases}
1&\mbox{ if }\Gamma_2=\Gamma_0(\varpi),\\
\varpi&\mbox{ if }\Gamma_2=\Gamma[\varpi\mathfrak{d}^{-1},\mathfrak{d}].
\end{cases}
\]
\begin{lemma}\label{repstsofthedoublecosets}
The sets 
\[
\left\{\mathbf{m}(\varpi^m)\,|\,m\in\mathbb{Z}_{\geq0}\right\}
\]
and
\[
\left\{\mathbf{m}(\varpi^m)\,|\,m\in\mathbb{Z}\right\}\cup\left\{\mathbf{w}_{\boldsymbol{\delta}\varpi^m}\,|\,m\in\mathbb{Z}\right\}
\]
form complete systems of the representatives for the double cosets in $\widetilde{\Gamma_1}\backslash \Mp_2(F)/\widetilde{\Gamma_1}$ and $\widetilde{\Gamma_2}\backslash \Mp_2(F)/\widetilde{\Gamma_2}$ respectively.
\end{lemma}
\begin{proof}
This follows simply from an easy argument on the row-column reductions.
\end{proof}
\begin{lemma}\label{decompstnofdoublecosets}
We can decompose the double cosets into left cosets as following.
\begin{enumerate}
\item For $m\geq1$ we have 
\begin{align*}
&\widetilde{\Gamma_1}\mathbf{m}(\varpi^m)\widetilde{\Gamma_1}\\
=&\left(\bigsqcup_{s\in\mathfrak{o}/\mathfrak{p}^{2m}}\mathbf{u}^\sharp(\boldsymbol{\delta}^{-1}s)\mathbf{m}(\varpi^m)\widetilde{\Gamma_1}\right)\bigsqcup
\left(\bigsqcup_{s\in\mathfrak{o}/\mathfrak{p}^{2m-1}}\mathbf{w}_{\boldsymbol{\delta}}\mathbf{u}^\sharp(\boldsymbol{\delta}^{-1}\varpi s)\mathbf{m}(\varpi^m)\widetilde{\Gamma_0(1)}\right).
\end{align*}
\item For $m\geq0$ we have 
\begin{align*}
\widetilde{\Gamma_2}\mathbf{m}(\varpi^m)\widetilde{\Gamma_2}
=\bigsqcup_{s\in\mathfrak{o}/\mathfrak{p}^{2m}}\mathbf{u}^\sharp(\boldsymbol{\delta}^{-1}\boldsymbol{\mu}s)\mathbf{m}(\varpi^m)\widetilde{\Gamma_2}.
\end{align*}
\item For $m\geq1$ we have 
\begin{align*}
\widetilde{\Gamma_2}\mathbf{m}(\varpi^{-m})\widetilde{\Gamma_2}
=\bigsqcup_{s\in\mathfrak{o}/\mathfrak{p}^{2m}}\mathbf{u}^\flat(\boldsymbol{\delta}\boldsymbol{\mu}^{-1}\varpi s)\mathbf{m}(\varpi^m)\widetilde{\Gamma_2}.
\end{align*}
\item For $m\geq1$ we have 
\begin{align*}
\widetilde{\Gamma_2}\mathbf{w}_{\boldsymbol{\delta}\boldsymbol{\mu}^{-1}\varpi^m}\widetilde{\Gamma_2}
=\bigsqcup_{s\in\mathfrak{o}/\mathfrak{p}^{2m-1}}\mathbf{u}^\flat(\boldsymbol{\delta}\boldsymbol{\mu}^{-1}\varpi s)\mathbf{w}_{\boldsymbol{\delta}\boldsymbol{\mu}^{-1}\varpi^{m}}\widetilde{\Gamma_2}.
\end{align*}
\item For $m\geq0$ we have 
\begin{align*}
\widetilde{\Gamma_2}\mathbf{w}_{\boldsymbol{\delta}\boldsymbol{\mu}^{-1}\varpi^{-m}}\widetilde{\Gamma_2}
=\bigsqcup_{s\in\mathfrak{o}/\mathfrak{p}^{2m+1}}\mathbf{u}^\sharp(\boldsymbol{\delta}^{-1}\boldsymbol{\mu}s)\mathbf{w}_{\boldsymbol{\delta}\boldsymbol{\mu}^{-1}\varpi^{-m}}\widetilde{\Gamma_2}.
\end{align*}
\end{enumerate}
\begin{proof}
All of these can be proved by applying the triangular decompositions. 
\end{proof}

\end{lemma}
Let $T$ be the subgroup of diagonal matrices in $\SL_2(\mathfrak{o})$.
If we consider the restriction of the character $\varepsilon$ to $\widetilde{T},$ then it can be extended to a character of the normalizer of $\widetilde{T}$ in $\Mp_2(F)$ by setting
\[
\varepsilon(\mathbf{m}(a))=\frac{\alpha_\phi(a)}{\alpha_\psi(1)}
\]
for $a\in F^\times$ and
\[
\varepsilon(\mathbf{w}_{\boldsymbol{\delta}})=1.
\]
From these we also have
\[
\varepsilon(\mathbf{w}_{\boldsymbol{\delta}a})=\alpha_\psi(\boldsymbol{\delta}a).
\]
\par
Now let us first consider the case of $\widetilde{\mathcal{H}}(\widetilde{\Gamma_1}\backslash \Mp_2(F)/\widetilde{\Gamma_1};\varepsilon)$.
Using this character given above, we can define the characteristic function $X_g$ of the double cosets $\widetilde{\Gamma_1}g\widetilde{\Gamma_1}$ in $\widetilde{\Gamma_1}\backslash \Mp_2(F)/\widetilde{\Gamma_1}$ by setting
\begin{equation}\label{eq:defofXg}
X_g(\gamma_1 g \gamma_2)=\varepsilon(\gamma_1)\varepsilon(g)\varepsilon(\gamma_2)
\end{equation}
where $g\in\left\{\mathbf{m}(\varpi^m)\,|\,m\in\mathbb{Z}_{\geq0}\right\}$ and $\gamma_1, \gamma_2\in\widetilde{\Gamma_1}$.
The next well-known result is very helpful for the calculations in this section.

\begin{lemma}\label{lem:helpful lemma}
For $g, h\in\widetilde{\mathcal{H}}(\widetilde{\Gamma_1}\backslash \Mp_2(F)/\widetilde{\Gamma_1};\varepsilon)$ such that $\Vol(\widetilde{\Gamma_1}g\widetilde{\Gamma_1})\Vol(\widetilde{\Gamma_1}h\widetilde{\Gamma_1})=\Vol(\widetilde{\Gamma_1}gh\widetilde{\Gamma_1}),$ we have $X_g\ast X_h=X_{gh}$.
\end{lemma}

\par
Set $\widetilde{\mathcal{T}}_m=q^{-m/2}X_{\mathbf{m}(\varpi^m)}$.
The next proposition is well-known.

\begin{proposition} \label{generatorsofH1}
We have
\[
\widetilde{\mathcal{T}}_1^2=q+1+\widetilde{\mathcal{T}}_2
\]
and
\[
\widetilde{\mathcal{T}}_1\ast\widetilde{\mathcal{T}}_m=q\widetilde{\mathcal{T}}_{m-1}+\widetilde{\mathcal{T}}_{m+1}
\]
for $m\geq2$.
In particular, the Hecke algebra $\widetilde{H}(\widetilde{\Gamma}_0(1)\backslash \Mp_2(F)/\widetilde{\Gamma}_0(1);\varepsilon)$ is generated by $\widetilde{\mathcal{T}}_1$.
\end{proposition}
\begin{proof}
For any $g\in \Mp_2(F),$ by (1) of Lemma \ref{decompstnofdoublecosets}, we have
\begin{align*}
 \widetilde{\mathcal{T}}_1\ast\widetilde{\mathcal{T}}_m(g)=
 &q^{-1/2}\frac{\alpha_\psi(1)}{\alpha_\psi(\varpi)}\left(\sum_{s\in\mathfrak{o}/\mathfrak{p}^2}\widetilde{\mathcal{T}}_m(\mathbf{m}(\varpi^{-1})\mathbf{u}^\sharp(-\boldsymbol{\delta}^{-1}s)g)\right.\\
+&\left.\sum_{s\in\mathfrak{o}/\mathfrak{p}}\widetilde{\mathcal{T}}_m(\mathbf{m}(\varpi^{-1})\mathbf{u}^\sharp(-\boldsymbol{\delta}^{-1}\varpi s)\mathbf{w}_{-\boldsymbol{\delta}}g)\right).
\end{align*}
To determine $\widetilde{\mathcal{T}}_1\ast\widetilde{\mathcal{T}}_m,$ we only have to check its values at $g=\mathbf{m}(\varpi^r)$ for $r\geq0$.
\par
We first treat the former summation.
\par
If $r=0,$ we have
\[
\mathbf{m}(\varpi^{-1})\mathbf{u}^\sharp(-\boldsymbol{\delta}^{-1}s)=\mathbf{w}_{\boldsymbol{\delta}}\mathbf{m}(\varpi)\mathbf{w}_{-\boldsymbol{\delta}}\mathbf{u}^\sharp(-\boldsymbol{\delta}^{-1}s)\in\widetilde{\Gamma_1}\mathbf{m}(\varpi)\widetilde{\Gamma_1}
\]
for any $s\in\mathfrak{o}$.
\par
Similarly, if $r>0,$ we have
\begin{align*}
&\mathbf{m}(\varpi^{-1})\mathbf{u}^\sharp(-\boldsymbol{\delta}^{-1}s)\mathbf{m}(\varpi^r)\\
=&\mathbf{u}^\flat(-\boldsymbol{\delta}^{-1}\varpi^2s^{-1})\mathbf{w}_{\boldsymbol{\delta}s^{-1}}\mathbf{m}(\varpi^{r+1})\mathbf{u}^\flat(-\boldsymbol{\delta}^{-1}\varpi^{2r}s^{-1})\times(\varpi,\varpi^r)\\
\in&\widetilde{\Gamma_1}\mathbf{m}(\varpi^{r+1})\widetilde{\Gamma_1}
\end{align*}
for $s\in\mathfrak{o}^{\times},$
\begin{align*}
&\mathbf{m}(\varpi^{-1})\mathbf{u}^\sharp(-\boldsymbol{\delta}^{-1}s)\mathbf{m}(\varpi^r)\\
=&\mathbf{u}^\flat(-\boldsymbol{\delta}^{-1}\varpi^2s^{-1})\mathbf{w}_{\boldsymbol{\delta}\varpi s^{-1}}\mathbf{m}(\varpi^{r})\mathbf{u}^\flat(-\boldsymbol{\delta}^{-1}\varpi^{2r}s^{-1})\times(\varpi,\boldsymbol{\delta}s)\\
\in&\widetilde{\Gamma_1}\mathbf{m}(\varpi^{r})\widetilde{\Gamma_1}
\end{align*}
for $s\in\mathfrak{p}\backslash\mathfrak{p}^2$ and
\begin{align*}
&\mathbf{m}(\varpi^{-1})\mathbf{u}^\sharp(-\boldsymbol{\delta}^{-1}s)\mathbf{m}(\varpi^r)\\
=&\mathbf{u}^\sharp(-\boldsymbol{\delta}^{-1}\varpi^{-2}s)\mathbf{m}(\varpi^{r-1})\times(\varpi,\varpi^r)\\
\in&\widetilde{\Gamma_1}\mathbf{m}(\varpi^{r-1})\widetilde{\Gamma_1}
\end{align*}
for $s\in\mathfrak{p}^2.$
\par
From above we get that if $m=1,$ then the former summation may occur only if $r=0, 1$ or $2$ and in those cases, we have
\[
\sum_{s\in\mathfrak{o}/\mathfrak{p}^2}\widetilde{\mathcal{T}}_1(\mathbf{m}(\varpi^{-1})\mathbf{u}^\sharp(-\boldsymbol{\delta}^{-1}s))=q^{3/2}\frac{\alpha_\psi(\varpi)}{\alpha_\psi(1)},
\]
\[
\sum_{s\in\mathfrak{o}/\mathfrak{p}^2}\widetilde{\mathcal{T}}_1(\mathbf{m}(\varpi^{-1})\mathbf{u}^\sharp(-\boldsymbol{\delta}^{-1}s)\mathbf{m}(\varpi))=0
\]
by Lemma \ref{lem:sumofWeilconsts} or
\[
\sum_{s\in\mathfrak{o}/\mathfrak{p}^2}\widetilde{\mathcal{T}}_1(\mathbf{m}(\varpi^{-1})\mathbf{u}^\sharp(-\boldsymbol{\delta}^{-1}s)\mathbf{m}(\varpi^2))=q^{-1/2}\frac{\alpha_\psi(\varpi)}{\alpha_\psi(1)},
\]
respectively.
Similarly, if $m\geq2,$ then the former summation may occur only if $r=m-1, m$ or $m+1$ and in those cases, we have
\[
\sum_{s\in\mathfrak{o}/\mathfrak{p}^2}\widetilde{\mathcal{T}}_m(\mathbf{m}(\varpi^{-1})\mathbf{u}^\sharp(-\boldsymbol{\delta}^{-1}s)\mathbf{m}(\varpi^{m-1}))=(q^{2-m/2}-q^{1-m/2})\frac{\alpha_\psi(\varpi)\alpha_\psi(\varpi^{m-1})}{\alpha_\psi(1)^2},
\]
\[
\sum_{s\in\mathfrak{o}/\mathfrak{p}^2}\widetilde{\mathcal{T}}_m(\mathbf{m}(\varpi^{-1})\mathbf{u}^\sharp(-\boldsymbol{\delta}^{-1}s)\mathbf{m}(\varpi^{m}))=0
\]
or
\[
\sum_{s\in\mathfrak{o}/\mathfrak{p}^2}\widetilde{\mathcal{T}}_m(\mathbf{m}(\varpi^{-1})\mathbf{u}^\sharp(-\boldsymbol{\delta}^{-1}s)\mathbf{m}(\varpi^{m+1}))=q^{-m/2}\frac{\alpha_\psi(\varpi)\alpha_\psi(\varpi^{m+1})}{\alpha_\psi(1)^2},
\]
respectively.
\par
Secondly, we treat the latter summation, which is easier.
We have
\begin{align*}
&\sum_{s\in\mathfrak{o}/\mathfrak{p}}\widetilde{\mathcal{T}}_m(\mathbf{m}(\varpi^{-1})\mathbf{u}^\sharp(-\boldsymbol{\delta}^{-1}\varpi s)\mathbf{w}_{-\boldsymbol{\delta}}\mathbf{m}(\varpi^r))\\
=&\mathbf{w}_{-\boldsymbol{\delta}}\mathbf{m}(\varpi^{m+1})\mathbf{u}^\flat(\boldsymbol{\delta}\varpi^{2r+1}s)\times(\varpi,\varpi^r)\\
\in&\widetilde{\Gamma_1}\mathbf{m}(\varpi^{r+1})\widetilde{\Gamma_1}
\end{align*}
for any $r\geq0$ and $s\in\mathfrak{o}$.
Thus the latter summation may occur only if $r=m-1$ in which case we have
\[
\sum_{s\in\mathfrak{o}/\mathfrak{p}}\widetilde{\mathcal{T}}_m(\mathbf{m}(\varpi^{-1})\mathbf{u}^\sharp(-\boldsymbol{\delta}^{-1}\varpi s)\mathbf{w}_{-\boldsymbol{\delta}}\mathbf{m}(\varpi^{m-1}))=q^{1-m/2}\frac{\alpha_\psi(\varpi)\alpha_\psi(\varpi^{m-1})}{\alpha_\psi(1)^2}.
\]
Concluding these results, we get
\[
\widetilde{\mathcal{T}}_1^2=q+1+\widetilde{\mathcal{T}}_2
\]
and
\[
\widetilde{\mathcal{T}}_1\ast\widetilde{\mathcal{T}}_m=q\widetilde{\mathcal{T}}_{m-1}+\widetilde{\mathcal{T}}_{m+1}
\]
for $m\geq2$.
\par
Now the fact that the Hecke algebra is generated by $\widetilde{\mathcal{T}}_1$ can be easily deduced by the induction.
\end{proof}
Remind that if we set
\[
K'=\left\{
\begin{pmatrix}
a&b\\c&d
\end{pmatrix}\in \PGL_2(F)\,\bigg|\,a, d\in\mathfrak{o}, b\in\mathfrak{d}^{-1}, c\in\mathfrak{d}
\right\}
\]
to be the standard maximal open compact subgroup of $\PGL_2(F)$ and $\mathcal{T}_m$ to be the characteristic function of
\[
K'\begin{pmatrix}
\varpi^m&0\\0&1
\end{pmatrix}K'
\]
for $m\geq0$.
Then the $\mathcal{T}_m$'s satisfy the same relations as the ones in Proposition \ref{generatorsofH1} and $\mathcal{T}_1$ generates the Hecke algebra $\mathcal{H}(K'\backslash \PGL_2(F)/K')$.
Hence we get the next corollary.
\begin{corollary}
The Hecke algebras $\widetilde{\mathcal{H}}(\widetilde{\Gamma_1}\backslash \Mp_2(F)/\widetilde{\Gamma_1};\varepsilon)$ and \\ $\mathcal{H}(K'\backslash \PGL_2(F)/K')$ are isomorphic $\mathbb{C}$-isomorphic algebras.
\end{corollary}

\begin{remark}
	The structure of $\mathcal{H}(K'\backslash\PGL_2(F)/K')$ given before the corollary also holds for the case $q$ is even.
\end{remark}

Now we want to investigate the structure of $\widetilde{\mathcal{H}}(\widetilde{\Gamma_2}\backslash \Mp_2(F)/\widetilde{\Gamma_2};\varepsilon)$.
Similarly to the previous case, for any representative $g$ in $\widetilde{\Gamma_2}\backslash \Mp_2(F)/\widetilde{\Gamma_2},$ we put $X_g$ to be the characteristic function of the double coset $\widetilde{\Gamma_2}g\widetilde{\Gamma_2}$ by setting 
\[
X_g(\gamma_1g\gamma_2)=\varepsilon(\gamma_1)\varepsilon(g)\varepsilon(\gamma_2)
\]
for $\gamma_1, \gamma_2\in\widetilde{\Gamma_0(\varpi)}$.
\par
Set
\[
\widetilde{\mathcal{T}}_m=\begin{cases}
q^{-|m|_\infty}X_{\mathbf{m}(\varpi^m)}&\mbox{ if }\Gamma_2=\Gamma_0(1),\\
q^{-|m|_\infty}X_{\mathbf{m}(\varpi^{-m})}&\mbox{ if }\Gamma_2=\Gamma[\varpi\mathfrak{d}^{-1},\mathfrak{d}],
\end{cases}
\]
and 
\[
\widetilde{\mathcal{U}}_m=\begin{cases}
q^{-|m|_\infty}X_{\mathbf{w}_{\boldsymbol{\delta}\varpi^m}}&\mbox{ if }\Gamma_2=\Gamma_0(1),\\
q^{-|m|_\infty}X_{\mathbf{w}_{\boldsymbol{\delta}\varpi^{-m}}}&\mbox{ if }\Gamma_2=\Gamma[\varpi\mathfrak{d}^{-1},\mathfrak{d}]
\end{cases}
\]
for $m\in\mathbb{Z}$.
Note that $|\cdot|_\infty$ is the usual absolute value for the real number.
To understand the structure of $\widetilde{\mathcal{H}}(\widetilde{\Gamma_2}\backslash \Mp_2(F)/\widetilde{\Gamma_2};\varepsilon),$ we need the next well-known lemma.
\begin{lemma} \label{generatorsofH}
We have the following equations:
\begin{enumerate}
\item $\widetilde{\mathcal{U}}_0\ast\widetilde{\mathcal{U}}_1=\widetilde{\mathcal{T}}_1$ and $\widetilde{\mathcal{U}}_1\ast\widetilde{\mathcal{U}}_0=\widetilde{\mathcal{T}}_{-1}$.
\item $\widetilde{\mathcal{T}}_{m_1}\ast \widetilde{\mathcal{T}}_{m_2}=\widetilde{\mathcal{T}}_{m_1+m_2}$ for $m_1$ and $m_2$ such that $m_1m_2\geq0$.
\item $\widetilde{\mathcal{U}}_1\ast\widetilde{\mathcal{T}}_{m}=\widetilde{\mathcal{U}}_{m+1}$ and $\widetilde{\mathcal{U}}_0\ast\widetilde{\mathcal{T}}_{-m}=\widetilde{\mathcal{U}}_{-m}$ for $m\geq0$.
\item $\widetilde{\mathcal{U}}_0^2=(q-1)\widetilde{\mathcal{U}}_0+q$ and $\widetilde{\mathcal{U}}_1^2=1$.
\end{enumerate}
\end{lemma}
\begin{proof}
One uses Lemma \ref{lem:helpful lemma} for (1) to (3) and similar calculations in the proof of Proposition \ref{generatorsofH1} for (4).
\end{proof}
The operator $\widetilde{\mathcal{U}}_1$ has eigenvalues $1$ and $-1$.
It is called the Atkin-Lehner involution.
By Lemma \ref{generatorsofH}, we get the structure of the Hecke algebra.
\begin{proposition}
The Hecke algebra $\widetilde{\mathcal{H}}(\widetilde{\Gamma_2}\backslash \Mp_2(F)\widetilde{\Gamma_2};\varepsilon)$ is generated by $\widetilde{\mathcal{U}}_0$ and $\widetilde{\mathcal{U}}_1$ with the relations:
\begin{enumerate}
\item $(\widetilde{\mathcal{U}}_0-q)(\widetilde{\mathcal{U}}_0+1)=0$
\item $\widetilde{\mathcal{U}}_1^2=1$
\end{enumerate}
\end{proposition}\label{generatorsfortheHeckealgebra}
Also, we can get the similar result for the case of $\PGL_2$.
Let $I$ be the Iwahori subgroup of $\PGL_2(F)$ which is given by
\[
I=\left\{
\begin{pmatrix}
a&b\\\varpi c&d
\end{pmatrix}\in\PGL_2(F)\,\bigg|\,a ,d\in\mathfrak{o}, b\in\mathfrak{d}^{-1}, c\in\mathfrak{d}
\right\}
\]
and $\mathcal{T}_m$ and $\mathcal{U}_m$ be the characteristic functions of
\[
I\begin{pmatrix}
\varpi^m&0\\0&1
\end{pmatrix}I
\mbox{ and }
\begin{pmatrix}
0&\boldsymbol{\delta}^{-1}\\\varpi^m\boldsymbol{\delta}&0
\end{pmatrix}I,
\]
respectively.
Then these characteristic functions satisfy the same relations stated in Lemma \ref{generatorsofH} and $\mathcal{U}_0$ and $\mathcal{U}_1$ generate the Hecke algebra $\mathcal{H}(I\backslash\PGL_2(F)/I)$.
Thus by Proposition \ref{generatorsfortheHeckealgebra}, we get the following corollary.
\begin{corollary}
	The Hecke algebras $\widetilde{H}(\widetilde{\Gamma_2}\backslash \Mp_2(F)/\widetilde{\Gamma_2};\varepsilon)$ and \\ $\mathcal{H}(I\backslash \PGL_2(F)/I)$ are isomorphic $\mathbb{C}$-isomorphic algebras.
\end{corollary}

\par

The case that $q$ is even is much more complicated.
We just state the results we need in the later sections.
Note that now $K=\Gamma_0(1)$ and $\Gamma=\Gamma_0(4)$.

Again, by the triangular decomposition, we have the following lemma.
\begin{lemma}\label{lem:decompofdoublecosetsveven}
For $q$ even and $\Gamma=\Gamma_0(4),$ we have
\[
\widetilde{\Gamma}\mathbf{m}(\varpi^m)\widetilde{\Gamma}=\bigsqcup_{s\in\mathfrak{o}/\mathfrak{p}^{2m}}\mathbf{u}^\sharp(\boldsymbol{\delta}^{-1}s)\mathbf{m}(\varpi^m)\widetilde{\Gamma}
\]
\end{lemma}
Put
\[
\varepsilon(\mathbf{m}(a))=\frac{\alpha_\phi(a)}{\alpha_\psi(1)}
\]
for $a\in F^\times,$
\[
\varepsilon(\mathbf{w}_{\boldsymbol{\delta}})=\overline{\alpha_\psi(\boldsymbol{\delta})},
\]
and
\[
\varepsilon(\mathbf{u}^\flat(\boldsymbol{\delta}c))=1
\]
for $c\in\mathfrak{o}$.
As in the previous cases, for any element $g\in\Mp_2(F)$ at which $\varepsilon$ is defined, let $X_g$ be the characteristic function of $\widetilde{\Gamma_0(4)}g\widetilde{\Gamma_0(4)}$ as (\ref{eq:defofXg}).
For $m\in\mathbb{Z},$ put $\widetilde{\mathcal{T}}_m=q^{-|m|_\infty/2}X_{\mathbf{m}(\varpi^m)}$.
By similar calculations as in the previous cases, we have the following lemma.

\begin{lemma}
We have $\widetilde{\mathcal{T}}_{m_1}\ast\widetilde{\mathcal{T}}_{m_2}=\widetilde{\mathcal{T}}_{m_1+m_2}$ for $m_1m_2\geq0$.
\end{lemma}

\section{The Hecke operators}\label{The Hecke operators}

In this section we want to take a look at how some of the Hecke operators in $\widetilde{\mathcal{H}}(\widetilde{\Gamma}\backslash\Mp_2(F)/\widetilde{\Gamma};\varepsilon)$ act on the "$\Gamma$-fixed" vectors in the principal series $\tilde{I}_{\psi}(s),$ especially for the case $q$ is odd.
We apply the notations used in the last section.

\begin{definition}
A vector $f$ in the principal series representation $\tilde{I}_{\psi}(s)$ is called $\Gamma$-fixed if it satisfies
\[
\rho(\gamma)f=\varepsilon(\gamma)^{-1}f
\]
for any $\gamma\in\widetilde{\Gamma}$.
We denote the subspace of $\tilde{I}_{\psi}(s)$ consisting of the $\Gamma$-fixed vectors by $\tilde{I}_{\psi}(s)^{\Gamma}$.
\end{definition}

Now we assume that $q$ is odd.
By Iwasawa decomposition, any vector in $\tilde{I}_{\psi}(s)$ can be determined by its values on $\widetilde{\Gamma_0(1)}$.
Thus if $\Gamma=\Gamma_1,$ the space $\tilde{I}_{\psi}(s)^{\Gamma}$ is one-dimensional and is spanned by the particular function $f_0\in\tilde{I}_{\psi}(s)^{\Gamma},$ which satisfies
\[
f_0(I)=1.
\]
If $\Gamma=\Gamma_2,$ since 
\[
\widetilde{B}\backslash\Mp_2(F)/\widetilde{\Gamma_2}=\left\{I,\mathbf{w}_{\boldsymbol{\delta}}\right\},
\]
the space $\tilde{I}_{\psi}(s)^{\Gamma}$ is two-dimensional and is spanned by the two particular functions $f_1$ and $f_2\in\tilde{I}_{\psi}(s)^{\Gamma}$ which satisfy
\[
f_1(I)=
\begin{cases}
1&\mbox{ if }\Gamma=\Gamma_0(\varpi),\\
0&\mbox{ if }\Gamma=\Gamma[\varpi\mathfrak{d}^{-1},\mathfrak{d}]
\end{cases}
\]
and
\[
f_2(\mathbf{w}_{\boldsymbol{\delta}})=
\begin{cases}
0&\mbox{ if }\Gamma=\Gamma_0(\varpi),\\
1&\mbox{ if }\Gamma=\Gamma[\varpi\mathfrak{d}^{-1},\mathfrak{d}]
\end{cases}
\]
respectively.
It is obvious that $\tilde{I}_{\psi}(s)^{\Gamma}$ is left invariant under the action of any Hecke operator in $\widetilde{\mathcal{H}}(\widetilde{\Gamma}\backslash\Mp_2(F)/\widetilde{\Gamma};\varepsilon)$.

\par

Note that for $g\in\Mp_2(F)$ at which $X_g$ is defined such that $\widetilde{\Gamma}g\widetilde{\Gamma}=\bigsqcup_{i}g_i\widetilde{\Gamma}$ and $f\in\tilde{I}_{\psi}(s)^{\Gamma},$ we have
\begin{align*}
 &\rho(X_g)f\\
=&\int_{\widetilde{\Gamma}g\widetilde{\Gamma}}X_g(h)\rho(h)dh\\
=&\sum_{i}\int_{\widetilde{\Gamma}}X_g(g_i\gamma)\rho(g_i\gamma)fd\gamma\\
=&\sum_{i}X_g(g_i)\rho(g_i)f.
\end{align*}
Using this equation we can complete most of the calculations in this section.

\par

First let us consider the case $\Gamma=\Gamma_1$.
Note that by Lemma \ref{generatorsofH1}, the Hecke algebra $\widetilde{\mathcal{H}}(\widetilde{\Gamma}\backslash\Mp_2(F)/\widetilde{\Gamma};\varepsilon)$ is generated by a single operator $\widetilde{\mathcal{T}}_1$.

\begin{lemma}\label{lem:action of T1, gamma1}
For the case $\Gamma=\Gamma_1,$ we have
\[
\rho(\widetilde{\mathcal{T}}_1)f_0=q^{1/2}(q^s+q^{-s})f_0.
\]
\end{lemma}
\begin{proof}
Since $\tilde{I}_{\psi}(s)^{\Gamma}$ is one-dimensional, we know that $f_0$ is an eigenvector of $\widetilde{\mathcal{T}}_1$.
Thus to get the eigenvalue, it suffices to calculate the value of $\widetilde{\mathcal{T}}_1$ at $I$.
By (1) of Lemma \ref{decompstnofdoublecosets}, one has
\begin{align*}
 &\rho(\widetilde{\mathcal{T}}_1)f_0(I)\\
=&q^{-1/2}\frac{\alpha_{\psi}(\varpi)}{\alpha_{\psi}(1)}\left(\sum_{\xi\in\mathfrak{o}/\mathfrak{p}^2}f_0(\mathbf{u}^\sharp(\boldsymbol{\delta}^{-1}\xi)\mathbf{m}(\varpi))+\sum_{\xi\in\mathfrak{o}/\mathfrak{p}}f_0(\mathbf{w}_{\boldsymbol{\delta}}\mathbf{u}^\sharp(\boldsymbol{\delta}^{-1}\varpi\xi)\mathbf{m}(\varpi))\right)\\
\end{align*}
For the first summand it is easy to get
\[
\sum_{\xi\in\mathfrak{o}/\mathfrak{p}^2}f_0(\mathbf{u}^\sharp(\boldsymbol{\delta}^{-1}\xi)\mathbf{m}(\varpi))=q^{1-s}\frac{\alpha_{\psi}(1)}{\alpha_{\psi}(\varpi)}.
\]
On the other hand, for the second summand, note that
\begin{align*}
 &\mathbf{w}_{\boldsymbol{\delta}}\mathbf{u}^\sharp(\boldsymbol{\delta}^{-1}\varpi\xi)\mathbf{m}(\varpi)\\
=&\begin{cases}
\mathbf{m}(\varpi^{-1})\mathbf{w}_{\boldsymbol{\delta}}\mathbf{u}^\sharp(\boldsymbol{\delta}^{-1}\varpi^{-1}\xi)\mbox{ if }\xi\in\mathfrak{p},\\
\mathbf{m}(\xi^{-1})\mathbf{u}^\sharp(-\boldsymbol{\delta}^{-1}\varpi^{-1}\xi)\mathbf{u}^\flat(\boldsymbol{\delta}\varpi\xi^{-1})\times(\boldsymbol{\delta} \xi,-\varpi\boldsymbol{\delta})\mbox{ if }\xi\in\mathfrak{o}^\times.
\end{cases}
\end{align*}
Hence 
\begin{align*}
 &\sum_{\xi\in\mathfrak{o}/\mathfrak{p}}f_0(\mathbf{w}_{\boldsymbol{\delta}}\mathbf{u}^\sharp(\boldsymbol{\delta}^{-1}\varpi\xi)\mathbf{m}(\varpi))\\
=&q^{s+1}\frac{\alpha_{\psi}(1)}{\alpha_{\psi}(\varpi)}+\sum_{\mathfrak{o}^\times/(1+\mathfrak{p})}\frac{\alpha_{\psi}(1)}{\alpha_{\psi}(\xi)}(\boldsymbol{\delta} \xi,-\varpi\boldsymbol{\delta})\\
=&q^{s+1}\frac{\alpha_{\psi}(1)}{\alpha_{\psi}(\varpi)}
\end{align*}
where for the second equation we have used the properties of Weil constants and Lemma \ref{lem:sumofWeilconsts}.
Applying the results for the two summands to the formula of $\rho(\widetilde{\mathcal{T}}_1)f_0(I),$ one gets the the lemma.
\end{proof}

Secondly, we have to consider the case $\Gamma=\Gamma_2$.
Now the Hecke algebra $\widetilde{\mathcal{H}}(\widetilde{\Gamma}\backslash\Mp_2(F)/\widetilde{\Gamma};\varepsilon)$ is generated by $\widetilde{\mathcal{U}}_0$ and $\widetilde{\mathcal{U}}_1$.
We want to investigate how both of the two generators act on $f_1$ and $f_2$.

\begin{lemma}\label{lem:action of operators, gamma2}
For the case $\Gamma=\Gamma_2,$ we have

\begin{align}
\label{0and1}    &\rho(\widetilde{\mathcal{U}}_0)f_1=f_2,\\
\label{0and2}    &\rho(\widetilde{\mathcal{U}}_0)f_2=qf_1+(q-1)f_2,\\
\label{1and1}    &\rho(\widetilde{\mathcal{U}}_1)f_1=q^{-1/2-s}f_2,\\
\label{1and2}	 &\rho(\widetilde{\mathcal{U}}_1)f_2=q^{1/2+s}f_1.
\end{align}

\end{lemma}
\begin{proof}
We only state the calculations for (\ref{0and1}) and (\ref{0and2}) for the case $\Gamma=\Gamma_0(\varpi)$ since the others are similar.
Because $\rho(\widetilde{\mathcal{U}}_0)f_i$ is also fixed by $\Gamma,$ we only have to get its values at $I$ and $\mathbf{w}_{\boldsymbol{\delta}}$.
By Lemma \ref{decompstnofdoublecosets} we have
\[
\rho(\widetilde{\mathcal{U}}_0)f_i(g)=\sum_{\xi\in\mathfrak{o}/\mathfrak{p}}f_i(g\mathbf{u}^\sharp(\boldsymbol{\delta}^{-1}\xi)\mathbf{w}_{\boldsymbol{\delta}})\quad(i=1, 2).
\]
Now for $g=I,$ we have that 
\[
\mathbf{u}^\sharp(\boldsymbol{\delta}^{-1}\xi)\mathbf{w}_{\boldsymbol{\delta}}\in\mbox{Supp}(f_2)
\]
for any $\xi\in\mathfrak{o}$.
On the other hand, for $g=\mathbf{w}_{\boldsymbol{\delta}},$ if $\xi\in\mathfrak{p},$ then 
\[
\mathbf{w}_{\boldsymbol{\delta}}\mathbf{u}^\sharp(\boldsymbol{\delta}^{-1}\xi)\mathbf{w}_{\boldsymbol{\delta}}=\mathbf{m}(-1)\mathbf{u}^\flat(-\boldsymbol{\delta}^{-1}\xi)\times(\boldsymbol{\delta},-1)\in\mbox{Supp}(f_1).
\]
If $\xi\in\mathfrak{o}^\times,$ then
\begin{align*}
 &\mathbf{w}_{\boldsymbol{\delta}}\mathbf{u}^\sharp(\boldsymbol{\delta}^{-1}\xi)\mathbf{w}_{\boldsymbol{\delta}}\\
=&\mathbf{m}(\xi^{-1})\mathbf{u}^\sharp(-\boldsymbol{\delta}^{-1}\xi)\mathbf{w}_{\boldsymbol{\delta}}\mathbf{u}^\sharp(-\boldsymbol{\delta}^{-1}\xi^{-1})\times(\boldsymbol{\delta},\xi)\in\mbox{Supp}(f_2).
\end{align*}
Using the decompositions above, we have
\[
\rho(\widetilde{\mathcal{U}}_0)f_1(I)=0
\]
and
\[
\rho(\widetilde{\mathcal{U}}_0)f_1(\mathbf{w}_{\boldsymbol{\delta}})=1.
\]
Thus we get (\ref{0and1}).
\par
Similarly, we have
\[
\rho(\widetilde{\mathcal{U}}_0)f_2(I)=q
\]
and
\[
\rho(\widetilde{\mathcal{U}}_0)f_2(\mathbf{w}_{\boldsymbol{\delta}})=q-1,
\]
from which we get (\ref{0and2}).
\par
The equations (\ref{1and1}) and (\ref{1and2}) can be gotten by similar calculations.

\end{proof}

%%%%%%%%%%%%%%%%%%%%%%%%%%%%%%%%%%%%%%%%%%%%%%%%%%%%%%%%%%%%%%%%%%%%%%%%%%%%%%%%%%%%%%%%%%%%%%%%%%%%%%%%%%%%%%%%%%%%%%%%%%%%%%%%%%%%%%%%%%%%%%%%%%%%%%%%%%%%%%%%%%%%%%%%%%%%%%%%%%%%%%%%%%%%%%%%%%%%%%%%%%%%%%%

\section{The archimedean case}\label{The archimedean case}
If $F=\mathbb{R}$ and $\psi(x)=\mathbf{e}(x),$ the Weil constant $\alpha_\psi(x)$ is given by
\[
\alpha_\psi(x)=
\begin{cases}
\exp(\pi\sqrt{-1}/4)&\mbox{ if }x>0,\\
\exp(-\pi\sqrt{-1}/4)&\mbox{ if }x<0.
\end{cases}
\]
In this case, the real metaplectic group $\Mp_2(\mathbb{R})$ is the unique non-trivial topological double covering of $\SL_2(\mathbb{R})$.
The unique factor of automorphy $\tilde{j}$ on $\Mp_2(\mathbb{R})\times\mathfrak{h}$ satisfying
\[
\tilde{j}\left(\left[
\begin{pmatrix}
a&b\\c&d
\end{pmatrix},\zeta
\right]\tau\right)^2=c\tau+d
\] 
is given by
\[
\tilde{j}\left(\left[
\begin{pmatrix}
a&b\\c&d
\end{pmatrix},\zeta
\right]\tau\right)=
\begin{cases}
\zeta\sqrt{d}&\mbox{ if }c=0, d>0,\\
-\zeta\sqrt{d}&\mbox{ if }c=0, d<0,\\
\sqrt{c\tau+d}&\mbox{ otherwise.}
\end{cases}
\]

%%%%%%%%%%%%%%%%%%%%%%%%%%%%%%%%%%%%%%%%%%%%%%%%%%%%%%%%%%%%%%%%%%%%%%%%%%%%%%%%%%%%%%%%%%%%%%%%%%%%%%%%%%%%%%%%%%%%%%%%%%%%%%%%%%%%%%%%%%%%%%%%%%%%%%%%%%

\section{Automorphic forms on $\Mp_2(\mathbb{A})$}\label{Automorphic forms on MpA}
In this section, we let $F$ be a totally real number field over $\mathbb{Q}$ with degree $n$.
The notations $\mathfrak{o},$ $\mathfrak{d}_1$ and $\mathbb{A}=\mathbb{A}_F$ denote the ring of integers, the different and the adele ring of $F,$ respectively.
For simplicity, from now, if the local case with respect to some local place $v$ of $F$ is under consideration, we use the same notations given 
in Section \ref{Weil representation} with a right lower subscript $v$.
Also, we use a right lower subscript $\mathrm{f}$ or $\infty$ to indicate the finite component or infinite component, respectively.
With respect to any non-archimedean place $v$ we fix a uniformizer $\varpi_v\in\mathfrak{o}_v$.
\par
We let $\psi_1=\prod_v\psi_{1,v}:\mathbb{A}/F\rightarrow$ be the unique non-trivial additive character of $\mathbb{A}$ which is trivial on $F$ and satisfies that for any archimedean place $v,$ the local component $\psi_{1,v}$ is given by $\psi_{1,v}(x)=\mathbf{e}(x)$.
So for any finite place $v$, the index of $\psi_{1,v},$ which we denote by $c_{1,v},$ is the exponent of the corresponding prime ideal $\mathfrak{p}_v$ in the prime decomposition of $\mathfrak{d}$.
The finite part $\prod_{v<\infty}\psi_{1,v}$ of $\psi_1$ is denoted by $\psi_{1,\mathrm{f}}$.
\par
If $\omega_{\psi_{1,v}}$ is the local Weil representation of $\Mp_2(F_v)$ with respect to the local character $\psi_{1,v}$ and $\alpha_{\psi_{1,v}}(\quad)$ is the corresponding Weil index, then it is known that for $a\in F^\times,$ we have
\[
\prod_{v\leq\infty}\alpha_{\psi_{1,v}}(a)=1.
\]
\par
Now let us define the global metaplectic group $\Mp_2(\mathbb{A})$.
Let $\Gamma_0(\xi)_v$ be the congruence subgroup given in Section \ref{Weil representation} for any finite place $v$ of $F$ and $0\neq\xi\in\mathfrak{o}_v$.
If $v$ is odd, there exists a unique conanical splitting over $\Gamma_0(1)_v$ in $\widetilde{\Gamma_0(1)_v}$.
The image of this splitting is also denoted by $\Gamma_0(1)$ and is the stablizer of $\phi_{0,v}$ via the local Weil representation $\omega_{\psi_{1,v}}$.
Now the global metaplectic double covering $\Mp_2(\mathbb{A})$ of $\SL_2(\mathbb{A})$ is the restricted direct product of all $\Mp_2(F_v)$ with respect to $\{\Gamma_0(1)_v\}_{v:\tiny{\mbox{odd}}}$ divided by $\{(\zeta_v)\in\prod_v\{\pm1\}\,|\,\prod_v\zeta_v=1\}$.
It is known that $\SL_2(F)$ can be canonically embedded into $\Mp_2(\mathbb{A})$.
The image is also denoted by $\SL_2(F)$.
Actually, through this embedding, any $\gamma\in \SL_2(F)$ is sent to an element quivalent to $[(\gamma_v)_v,(1)_v]\in \Mp_2(\mathbb{A})$.
For any subset $S$ of $\Mp_2(\mathbb{A}),$ the inverse image of $S$ in $\Mp_2(\mathbb{A})$ is denoted by $\widetilde{S}$.
\par
We let $\{\pm\mathbf{1}\}$ be the kernel of the surjective mapping $\Mp_2(\mathbb{A})\rightarrow \SL_2(\mathbb{A})$ where $\mathbf{1}$ is the identity element in $\Mp_2(\mathbb{A})$.
A function $\Phi$ on $\Mp_2(\mathbb{A})$ is called geniune if $\Phi(-\mathbf{1}g)=-\Phi(g)$.
Given a family of local geniune functions $(\Phi_v)_v$ where $\Phi_v$ is identical to $1$ on $\widetilde{\Gamma_0(1)_v}$ for almost all $v,$ by $(\prod_v\Phi_v)((g_v)_v)=\prod_v\Phi_v(g_v)$ we get a geniune function $\prod_v\Phi_v$ on $\Mp_2(\mathbb{A})$.
\par
Fix an $n$-tuple of non-negative integers $k=(k_1,\dots,k_n)\in\mathbb{Z}_{\geq0}^n$.
Let $\SL_2(\mathbb{A}_\mathrm{f})$ be the finite part of $\SL_2(\mathbb{A}),$ $\Gamma'_\mathrm{f}$ an open compact subgroup of $\SL_2(\mathbb{A}_\mathrm{f})$ and $\varepsilon'$ a geniune character from $\widetilde{\Gamma'_\mathrm{f}}$ to $\mathbb{C}^\times$.
If $\Gamma'=\SL_2(F)\cap(\Gamma'_\mathrm{f}\times \SL_2(\mathbb{R})^n),$ by $J^{k+1/2}_{\varepsilon'}$ we realize a factor of automorphy on $\Gamma'$ defined by
\begin{equation}\label{defofJ}
J^{k+1/2}_{\varepsilon'}(\gamma,z)=\varepsilon'([\gamma,1])\prod_{i=1}^n\tilde{j}([\iota_i(\gamma),1],z_i)^{2k_i+1}
\end{equation}
where $\iota_i$ are the $n$ real embeddings of $F$ and $z=(z_1,\dots,z_n)\in\mathfrak{h}^n$.
Then $M_{k+1/2}(\Gamma',\varepsilon)$ is the space of modular forms with the factor of automorphy $J^{k+1/2}_{\varepsilon'}$ and $S_{k+1/2}(\Gamma',\varepsilon)$ is the subspace of cusp forms in $M_{k+1/2}(\Gamma',\varepsilon)$.
\par
We want to lift a modular form $h\in M_{k+1/2}(\Gamma',\varepsilon')$ to an automorphic form $\Phi_h$ on $\SL_2(F)\backslash \Mp_2(\mathbb{A})$.
For $g\in \Mp_2(\mathbb{A}),$ by the strong approximation theorem, $g$ can be written in the form $g=\gamma g_\infty g_\mathrm{f}$ where $\gamma\in \SL_2(F),$ $g_\infty\in\widetilde{\SL_2(\mathbb{R})^n}$ and $g_\mathrm{f}\in\widetilde{\Gamma'_\mathrm{f}}$.
Now the value of $\Phi_h$ at $g$ is given by
\[
\Phi_h(g)=h(g_\infty(\mathbf{i}))\varepsilon'(g_\mathrm{f})^{-1}\prod_{i=1}^n\tilde{j}(\iota_i(g_\infty),\sqrt{-1})^{-2k_i-1}
\]
where $\mathbf{i}=(\sqrt{-1},\dots,\sqrt{-1})\in\mathfrak{h}^n$.
Note that $\widetilde{\SL_2(\mathbb{R})^n}$ acts on $\mathfrak{h}^n$ by the usual sence.
One can easily see that the definition does not depend on the choice of $\gamma$.
\par
We set
\[
\mathcal{A}_{k+1/2}(\SL_2(F)\backslash \Mp_2(\mathbb{A});\widetilde{\Gamma'_\mathrm{f}},\varepsilon')=\{\Phi_h\,|\,h\in M_{k+1/2}(\Gamma',\varepsilon')\}
\]
and
\[
\mathcal{A}^{\tiny\mbox{CUSP}}_{k+1/2}(\SL_2(F)\backslash \Mp_2(\mathbb{A});\widetilde{\Gamma'_\mathrm{f}},\varepsilon')=\{\Phi_h\,|\,h\in S_{k+1/2}(\Gamma',\varepsilon')\}.
\]
\par
On the other hand, for any $\Phi\in\mathcal{A}_{k+1/2}(\SL_2(F)\backslash \Mp_2(\mathbb{A});\widetilde{\Gamma'_\mathrm{f}},\varepsilon'),$ if we put
\[
h_\Phi(z)=\Phi(g_\infty)\prod_{i=1}^n\tilde{j}(\iota_i(g_\infty),\sqrt{-1})^{2k_i+1}
\]
where $z\in\mathfrak{h}^n$ and $g_\infty\in\widetilde{\SL_2(\mathbb{R})^n}$ is chosen so that $g_\infty(\mathbf{i})=z,$ then $h_\Phi\in M_{k+1/2}(\Gamma',\varepsilon')$.
Furthermore, using the notations above, we have
\[
h_{\Phi_h}=h,
\]
thus we get a bijection between $\mathcal{A}_{k+1/2}(\SL_2(F)\backslash \Mp_2(\mathbb{A});\widetilde{\Gamma'_\mathrm{f}},\varepsilon')$ and $M_{k+1/2}(\Gamma',\varepsilon')$.
The similar result holds for the cusp case.
\par
We set
\[
\mathcal{A}_{k+1/2}(\SL_2(F)\backslash \Mp_2(\mathbb{A}))=\bigcup_{(\Gamma'_\mathrm{f},\varepsilon')}\mathcal{A}_{k+1/2}(\SL_2(F)\backslash \Mp_2(\mathbb{A});\widetilde{\Gamma'_\mathrm{f}},\varepsilon')
\]
and
\[
\mathcal{A}^{\tiny\mbox{CUSP}}_{k+1/2}(\SL_2(F)\backslash \Mp_2(\mathbb{A}))=\bigcup_{(\Gamma'_\mathrm{f},\varepsilon')}\mathcal{A}^{\tiny\mbox{CUSP}}_{k+1/2}(\SL_2(F)\backslash \Mp_2(\mathbb{A});\widetilde{\Gamma'_\mathrm{f}},\varepsilon')
\]
where $(\Gamma'_\mathrm{f},\varepsilon')$ runs over all pairs of open compact subgroup $\Gamma'_\mathrm{f}\in \SL_2(\mathbb{A}_\mathrm{f})$ and geniune character $\varepsilon$ of $\Gamma'_\mathrm{f}$.
The finite part $\Mp_2(\mathbb{A}_\mathrm{f})=\widetilde{\SL_2(\mathbb{A}_\mathrm{f})}$ acts on $\mathcal{A}_{k+1/2}(\SL_2(F)\backslash \Mp_2(\mathbb{A}))$ by the right translation $\rho$.
Also, the corresponding action of $\Mp_2(\mathbb{A}_\mathrm{f})=\widetilde{\SL_2(\mathbb{A}_\mathrm{f})}$ on $\bigcup_{(\Gamma'_\mathrm{f},\varepsilon')}M_{k+1/2}(\Gamma',\varepsilon')$ is also denoted by $\rho$.
The next lemma is well-known and can be gotten easily by using the definition and Fourier expansion.
\begin{lemma}\label{behavioursofusharpandm}
If $f$ is a modular form of weight $k+1/2$ with Fourier expansion $f(z)=\sum_{\xi\in F}c(\xi)q^\xi,$ for any $x_\mathrm{f}\in\mathbb{A}_\mathrm{f}$ and totally positive $a\in F^\times,$ we have
\[
\rho(\mathbf{u}^\sharp(x_\mathrm{f}))f(z)=\sum_{\xi\in F}c(\xi)\psi_1(\xi x)q^\xi
\]
and
\[
\rho(\mathbf{m}(a_\mathrm{f}))f(z)=f(a^{-2}z)a^{-k-1/2}
\]
where $a^{-2k-1}=\prod_{i=1}^n\iota_i(a)^{-k_i-1/2}$.
\end{lemma}
\par
From now, let us fix an odd square-free integral ideal $\mathfrak{I}$.
The congruence subgroup $K=\Gamma[\mathfrak{d}_1^{-1},\mathfrak{I}\mathfrak{d}_1]$ of $\SL_2(F)$ is defined by
\[
K=
\left\{\begin{pmatrix}a&b\\c&d\end{pmatrix}\in \SL_2(F)\,\bigg|\,a,d\in\mathfrak{o}, b\in\mathfrak{d}_1^{-1}, c\in\mathfrak{I}\mathfrak{d}_1 
\right\}.
\]
Hence if $v$ is a finite place of $F,$ we have
\[
K_v=
\begin{cases}
\Gamma[\mathfrak{d}_{1,v}^{-1},\varpi_v\mathfrak{d}_{1,v}]_v&\mbox{ if }v\mid\mathfrak{I}\\
\Gamma[\mathfrak{d}_{1,v}^{-1},\mathfrak{d}_{1,v}]_v&\mbox{ otherwise.}
\end{cases}
\]
Similarly, we set $\Gamma=\Gamma[\mathfrak{d}_1^{-1},4\mathfrak{I}\mathfrak{d}_1]$.
Then for $v<\infty,$ we have
\[
\Gamma_v=
\begin{cases}
\Gamma[\mathfrak{d}_{1,v}^{-1},4\varpi_v\mathfrak{d}_{1,v}]_v=\Gamma[\mathfrak{d}_{1,v}^{-1},\varpi_v\mathfrak{d}_{1,v}]_v&\mbox{ if }v\mid\mathfrak{I}\\
\Gamma[\mathfrak{d}_{1,v}^{-1},4\mathfrak{d}_{1,v}]_v&\mbox{ otherwise.}
\end{cases}
\]
Let $\Gamma_\mathrm{f}=\prod_{v<\infty}\Gamma_v$.
\par
We fix an square-free integer $\mathfrak{f}$ such that
\[
\sgn\left(N_{F/\mathbb{Q}}(\mathfrak{f})\right)=(-1)^{\sum^n_{i=1} k_i}
\]
and
\[
\mathfrak{I}\subset(\mathfrak{f}).
\]
Here by saying $\mathfrak{f}$ is square-free we mean that the principal ideal $(\mathfrak{f})$ generated by $\mathfrak{f}$ is square-free.
Then we set $\psi$ to be the additive character of $\mathbb{A}$ given by
\[
\psi(x)=\psi_1(\mathfrak{f}x).
\]
For any finite place $v$ of $F,$ the degree of $\psi_v$ is $c_{1,v}+1$ if $v|\mathfrak{f}$ or $c_{1,v}$ otherwise.
We let $\mathfrak{d}=\mathfrak{f}\mathfrak{d}_1$.
Thus for any finite place $v,$ the maximal local ideal in $F_v$ on which $\psi_v$ is identical to $1$ is $\mathfrak{d}_v^{-1}$.
Set $\Gamma_0(\xi_v)_v=\Gamma[\mathfrak{d}_v^{-1},\xi_v\mathfrak{d}_v]_v$ for $\xi_v\in\mathfrak{o}_v$.
Then by the definitions of $K$ and $\Gamma$ given above, we have $K=\Gamma[\mathfrak{f}\mathfrak{d}^{-1},\mathfrak{I}\mathfrak{f}^{-1}\mathfrak{d}]$ and $\Gamma=\Gamma[\mathfrak{f}\mathfrak{d}^{-1},4\mathfrak{I}\mathfrak{f}^{-1}\mathfrak{d}]$.
In particular, for $v<\infty,$ we have that
\begin{equation}\label{defofKv}
	K_v=
	\begin{cases}
		\Gamma_0(\varpi_v)_v&\mbox{ if }v\mid\mathfrak{I}\mbox{ but }v\nmid\mathfrak{f},\\
		\Gamma[\varpi_v\mathfrak{d}_v^{-1},\mathfrak{d}_v]_v&\mbox{ if }v\mid\mathfrak{f},\\
		\Gamma_0(1)_v&\mbox{ otherwise,}
	\end{cases}
\end{equation}
and
\[
\Gamma_v=
\begin{cases}
K_v&\mbox{ if }v\mid\mathfrak{I}\\
\Gamma_0(4)_v&\mbox{ otherwise.}
\end{cases}
\]
\par
By Lemma $\ref{defofvarepsilon}$, for any finite place $v$ of $F,$ there exists a genuine character $\varepsilon_v:\widetilde{\Gamma_v}\rightarrow\mathbb{C}^\times$ given by $\omega_{\psi_v}(\gamma)\phi_{0,v}=\varepsilon_v(\gamma)^{-1}\phi_{0,v}$. 
From this we get a genuine character $\varepsilon=\prod_{v<\infty}\varepsilon_v$ of $\widetilde{\Gamma_\mathrm{f}}$.
We want to consider the factor of automorphy $J_\mathfrak{f}^{k+1/2}=J_{\varepsilon}^{k+1/2}$ defined as (\ref{defofJ}).
Set $J^{k+1/2}=J_1^{k+1/2}$.
It is known that if $\mathfrak{I}=\mathfrak{o}$ and $\mathfrak{f}=1,$ then $J^{1/2}$ is the usual factor of automorphy of weight $1/2$.
That is, it is the factor of automorphy $j^{1/2}$ on $\Gamma_0(4)\times\mathfrak{h}_n$ which satisfies
\[
\theta(\gamma z)=j^{1/2}(\gamma,z)\theta(z)
\]
for $\gamma\in\Gamma_0(4)$ and $z\in\mathfrak{h}^m$.
Here $\theta$ is the usual theta series given by
\[
\theta(z)=\sum_{\xi\in\mathfrak{o}}\mathbf{e}(\tr(\xi^2z))=\sum_{\xi\in\mathfrak{o}}q^{\xi^2}.
\]
If $\mathfrak{I}\neq\mathfrak{o}$ and $\mathfrak{f}$ satisfies $\sgn(N_{F/\mathbb{Q}}(\mathfrak{f}))=1,$ with the notations above, by the explicit formulas of $\varepsilon_v$ given in Section \ref{Weil representation} and the properties of Weil constants, we have
\[
J^{1/2}_\mathfrak{f}(\gamma,z)=j^{1/2}(\gamma,z)\prod_{v\,\mathrm{even}}(\mathfrak{f},d)_v\prod_{v\mid\mathfrak{f}}(\mathfrak{f},d)_v\mbox{ for }\gamma=
\begin{pmatrix}
a&b\\c&d
\end{pmatrix}\in\Gamma_0(4\mathfrak{I}),
\]
where $(\quad,\quad)_v$ is the quadratic Hilbert symbol for the completion $F_v$ of $F$ with respect to the place $v$.
For simplicity, we put
\[
\chi_\mathfrak{f}(d)=\prod_{v\,\mathrm{even}}(\mathfrak{f},d)_v\prod_{v\mid\mathfrak{f}}(\mathfrak{f},d)_v.
\]
Back to the case of weight $k=(k_1,\dots,k_n),$ the space of modular forms and cusp forms with respect to $J^{k+1/2}_\mathfrak{f}$ are denoted by $M_{k+1/2}(\Gamma_0(4\mathfrak{I}),\chi_\mathfrak{f})$ and $S_{k+1/2}(\Gamma_0(4\mathfrak{I}),\chi_\mathfrak{f}),$ respectively.
\par
For the case $F=\mathbb{Q}$ and weight $k+1/2,$ if $\mathfrak{I}$ is the principal ideal generated by some square-free odd natural number $N$ and $\mathfrak{f}=(-1)^kn$ for some positive divisor $n$ of $N,$ the space $S_{k+1/2}(\Gamma_0(4\mathfrak{I}),\chi_\mathfrak{f})$ coincides with the space $S_{k+1/2}(N,\left(\frac{}{n}\right))$ considered in \cite{Kohnen:82}.

%%%%%%%%%%%%%%%%%%%%%%%%%%%%%%%%%%%%%%%%%%%%%%%%%%%%%%%%%%%%%%%%%%%%%%%%%%%%%%%%%%%%%%%%%%%%%%%%%%%%%%%%%%%%%%%%%%%%%%%%%%%%%%%%%%%%%%%%%%%%%%%%%%%%%%%%%%%%%%%%%%%%%%%%%%%%%%%%%%%%%%%%%%%%%%%%%%%%%%%%%%%%%%%%%%%%%%%%%%%%%%%%%%%%%%%%%%

\section{The plus space}
We use the same notations as in Section \ref{The archimedean case}.
Now we want to define the plus spaces in $M_{k+1/2}(\Gamma_0(4\mathfrak{I}),\chi_\mathfrak{f})$ and  $S_{k+1/2}(\Gamma_0(4\mathfrak{I}),\chi_\mathfrak{f})$.
For any finite place $v$ of $F,$ let $K_v$ be the open compact congruence subgroup defined by (\ref{defofKv}) and the idempotent Hecke operator $E^K_v$ be the ones defined by Definition $\ref{defofeKandEK}$.
The global Hecke operator $E^K=\prod_{v<\infty}E^K_v$ acts on $\mathcal{A}_{k+1/2}(\SL_2(F)\backslash \Mp_2(\mathbb{A}))$ by
\[
\rho(E^K)\Phi(g)=\int_{\Mp_2(\mathbb{A}_\mathrm{f})}\rho(h)\Phi(g)E^K(h)dh.
\]
In the same way $E^K$ also acts on the space of all modular forms of weight $k+1/2$.
If we put
\[
\mathcal{A}_{k+1/2}(\SL_2(F)\backslash \Mp_2(\mathbb{A}))^{E^K}=\left\{\Phi\in\mathcal{A}_{k+1/2}(\SL_2(F)\backslash \Mp_2(\mathbb{A}))\,|\,\rho(E^K)\Phi=\Phi\right\}
\]
and
\[
\mathcal{A}^{\tiny\mbox{CUSP}}_{k+1/2}(\SL_2(F)\backslash \Mp_2(\mathbb{A}))^{E^K}=\left\{\Phi\in\mathcal{A}^{\tiny\mbox{CUSP}}_{k+1/2}(\SL_2(F)\backslash \Mp_2(\mathbb{A}))\,|\,\rho(E^K)\Phi=\Phi\right\},
\]
then obviously
\[
\mathcal{A}_{k+1/2}(\SL_2(F)\backslash \Mp_2(\mathbb{A}))^{E^K}\subset\mathcal{A}_{k+1/2}(\SL_2(F)\backslash \Mp_2(\mathbb{A});\widetilde{\Gamma_0(4\mathfrak{I})_\mathrm{f}},\varepsilon)
\]
and
\[
\mathcal{A}^{\tiny\mbox{CUSP}}_{k+1/2}(\SL_2(F)\backslash \Mp_2(\mathbb{A}))^{E^K}\subset\mathcal{A}^{\tiny\mbox{CUSP}}_{k+1/2}(\SL_2(F)\backslash \Mp_2(\mathbb{A});\widetilde{\Gamma_0(4\mathfrak{I})_\mathrm{f}},\varepsilon).
\]
This implies that
\[
M_{k+1/2}(\Gamma_0(4\mathfrak{I}),\chi_\mathfrak{f})^{E^K}\subset M_{k+1/2}(\Gamma_0(4\mathfrak{I}),\chi_\mathfrak{f})
\]
and
\[
S_{k+1/2}(\Gamma_0(4\mathfrak{I}),\chi_\mathfrak{f})^{E^K}\subset S_{k+1/2}(\Gamma_0(4\mathfrak{I}),\chi_\mathfrak{f}).
\]
\begin{definition}
We put
\[
M^+_{k+1/2}(\Gamma_0(4\mathfrak{I}),\chi_\mathfrak{f})=M_{k+1/2}(\Gamma_0(4\mathfrak{I}),\chi_\mathfrak{f})^{E^K}
\]
and
\[
S^+_{k+1/2}(\Gamma_0(4\mathfrak{I}),\chi_\mathfrak{f})=S_{k+1/2}(\Gamma_0(4\mathfrak{I}),\chi_\mathfrak{f})^{E^K}.
\]	
They are both called the Kohnen plus space.
\end{definition}
We are going to see that this definition coincides with the one defined by Kohnen or Hiraga \& Ikeda if $F=\mathbb{Q}$ or $\mathfrak{I}=\mathfrak{o},$ respectively.
For any finite place $v,$ fix one $\boldsymbol{\delta}_v\in F_v$ such that $\boldsymbol{\delta}_v\mathfrak{o}_v=\mathfrak{d}_v$.
\begin{proposition}\label{plusimpliesfourier}
If $f=\sum_{\xi\in F}c(\xi)q^\xi\in M^+_{k+1/2}(\Gamma_0(4\mathfrak{I}),\chi_\mathfrak{f}),$ then $c(\xi)\neq0$ only if there exists $\lambda\in\mathfrak{o}$ such that $\xi\equiv\mathfrak{f}\lambda^2\mod4\mathfrak{o}$.
\end{proposition}
\begin{proof}
The proof is very similar to the one for Proposition 13.4 in \cite{HiraIke:13}.
Let $\hat{\mathfrak{o}}=\prod_{v<\infty}\mathfrak{o}_v$ and $\boldsymbol\delta=\prod_{v<\infty}\boldsymbol{\delta}_v\in\mathbb{A}^\times$.
Put $e^K=\prod_{v<\infty}e^K_v$ where $e^K_v$ is as in Definition \ref{defofeKandEK} and $\mathbf{w}'_{2\boldsymbol{\delta}}=\prod_{v\mid2}\mathbf{w}_{2\boldsymbol{\delta}_v}$.
We set
\[
f_0=2^{\sum_{j=1}^n k_j}\prod_{v|2}\overline{\alpha_{\psi_v}(\boldsymbol{\delta}_v)}(2,\boldsymbol{\delta}_v)_v\times\rho({\mathbf{w}'_{2\boldsymbol{\delta}}}^{-1})f.
\]
Then $\rho(E^K)f=f$ is equivalent to
\[
\rho(e^K)f_0=f_0.
\]
Put $K_\mathrm{f}=\prod_{v<\infty}K_v$.
Since $e^K$ is a matrix coefficient of the irreducible representation $\Omega_\psi=\otimes_{v<\infty}\Omega_{\psi_v}$ of $\widetilde{K_\mathrm{f}}$ on $\mathbb{S}(2^{-1}\hat{\mathfrak{o}}/\hat{\mathfrak{o}})=\otimes_{v<\infty}\mathbb{S}(2^{-1}\mathfrak{o}_v/\mathfrak{o}_v),$ the above implies that the complex space $\mathcal{V}$ spanned by $\{\rho(\gamma)f_0\mid\gamma\in\widetilde{K_\mathrm{f}}\}$ is a representation of $\widetilde{K_\mathrm{f}}$ isomorphic to $\Omega_\psi$.
Here $\mathbb{S}(2^{-1}\hat{\mathfrak{o}}/\hat{\mathfrak{o}})$ consists of Schwartz functions $\phi$ on $2^{-1}\hat{\mathfrak{o}}$ such that $\phi(x+y)=\phi(x)$ for $y\in\hat{\mathfrak{o}}$.
By Chinese remainder theorem, if we denote the characteristic function of $\lambda/2+\hat{\mathfrak{o}}$ in $\mathbb{S}(2^{-1}\hat{\mathfrak{o}}/\hat{\mathfrak{o}})$ by $\phi_\lambda$ for any $\lambda\in\mathfrak{o}/2\mathfrak{o},$ then $\phi_\lambda$ form a orthonormal basis for $\mathbb{S}(2^{-1}\hat{\mathfrak{o}}/\hat{\mathfrak{o}})$.
Noticing that $\Omega_\psi(e^K)\phi_0=\phi_0$ by Schur's lemma, there exists an intertwining map $i:\mathcal{V}\rightarrow\mathbb{S}(2^{-1}\hat{\mathfrak{o}}/\hat{\mathfrak{o}})$ such that $i(f_0)=\phi_0$.
Likewise, for $\lambda\in\mathfrak{o}/2\mathfrak{o}$, we set $f_\lambda\in\mathcal{V}$ to be the functions such that
\[
i(f_\lambda)=\phi_\lambda.
\]
As $\mathbf{w}'_{2\boldsymbol{\delta}},$ we put $\mathbf{w}'_{\boldsymbol{\delta}}=\prod_{v\mid2}\mathbf{w}_{\boldsymbol{\delta}_v}$
Since
\[
\Omega_\psi(\mathbf{w}'_{\boldsymbol{\delta}})\phi_0=2^{-n/2}\prod_{v|2}\overline{\alpha_{\psi_v}(\boldsymbol{\delta}_v)}\sum_{\lambda\in\mathfrak{o}/2\mathfrak{o}}\phi_\lambda,
\]
we also have
\[
\rho(\mathbf{w}'_{\boldsymbol{\delta}})f_0=2^{-n/2}\prod_{v|2}\overline{\alpha_{\psi_v}(\boldsymbol{\delta}_v)}\sum_{\lambda\in\mathfrak{o}/2\mathfrak{o}}f_\lambda,
\]
from which we get
\begin{align}{\label{sumofflambda}}
\begin{split}
&\sum_{\lambda\in\mathfrak{o}/2\mathfrak{o}}f_\lambda(z)\\
=&2^{\sum_{j=1}^n (k_j+1/2)}\prod_{v\mid2}(2,\boldsymbol{\delta}_v)_v\times\rho(\mathbf{w}'_{\boldsymbol{\delta}}{\mathbf{w}'_{2\boldsymbol{\delta}}}^{-1})f(z)\\
=&2^{\sum_{j=1}^n (k_j+1/2)}\rho(\prod_{v\mid2}\mathbf{m}(2_v))f(z)\\
=&2^{\sum_{j=1}^n (k_j+1/2)}\rho(\mathbf{m}(2_v))f(z)\\
=&f(z/4)
\end{split}
\end{align}
where the last equation comes from the fact that $\rho(\mathbf{m}(2_v))$ acts trivially on $f$ for odd $v$ and Lemma \ref{behavioursofusharpandm}.
Put $\hat{\mathfrak{d}}=\boldsymbol{\delta}\hat{\mathfrak{o}}$.
Now since
\[
\Omega_\psi(\mathbf{u}^\sharp(x))\phi_\lambda=\psi(x\lambda^2/4)\phi_\lambda\quad(x\in\mathfrak{f}\hat{\mathfrak{d}}^{-1})
\]
we also have
\begin{equation}\label{usharpbehav}
\rho(\mathbf{u}^\sharp(x))f_\lambda=\psi(x\lambda^2/4)f_\lambda\quad(x\in\mathfrak{f}\hat{\mathfrak{d}}^{-1}).
\end{equation}
But if we write the the Fourier expansion of $f_\lambda$ in the form
\[
f_\lambda(z)=\sum_{\xi\in F}c_\lambda(\xi)q^{\xi/4},
\]
then by Lemma \ref{behavioursofusharpandm} again, we get
\[
\rho(\mathbf{u}^\sharp(x))f_\lambda(z)=\sum_{\xi\in F}c_\lambda(\xi)\psi_1(x\xi/4)q^{\xi/4}\quad(x\in\mathfrak{f}\hat{\mathfrak{d}}^{-1}).
\]
By comparing this with (\ref{usharpbehav}), we see $c_\lambda(\xi)=0$ unless $\psi(x\lambda^2/4)=\psi_1(x\xi/4)=\psi(\mathfrak{f}^{-1}x\xi/4)$ for all $x\in\mathfrak{f}\hat{\mathfrak{d}}^{-1},$ that is, $\xi\equiv\mathfrak{f}\lambda^2\mod4\mathfrak{o}$.
One can easily check that this does not depend on the choice of $\lambda\mod2\mathfrak{o}$.
Now by (\ref{sumofflambda}), we see that the $\xi$-th Fourier coefficient of $f$ does not vanish only if there exists some $\lambda\in\mathfrak{o}$ such that $\xi\equiv\mathfrak{f}\lambda^2\mod4\mathfrak{o},$ which is what we aimed to show.
\end{proof}
The inverse of this proposition is also true.
We can apply the proof used for Proposition 7.2 in \cite{Ren:16}.

\begin{proposition}\label{fourierimpliesplus}
If $f\in M_{k+1/2}(\Gamma_0(4\mathfrak{I}),\chi_\mathfrak{f})$ and the Fourier coefficients $c(\xi)$ of $f$ satisfy that $c(\xi)$ occurs only if there exists $\lambda\in\mathfrak{o}$ such that $\xi\equiv\mathfrak{f}\lambda^2\mod4\mathfrak{o},$ then $f\in M^+_{k+1/2}(\Gamma_0(4\mathfrak{I}),\chi_\mathfrak{f})$.
\end{proposition}
\begin{proof}
We need to show that $\rho(E^K)f=f$.
\par
Let $\mathcal{V}$ be the representation of $\Mp_2(\mathbb{A}_\mathrm{f})$ generated by $f(z/4)$ via $\rho$.
For $\lambda\in\mathfrak{o}/2\mathfrak{o},$ we put
\[
f_\lambda(z)=\sum_{\xi\equiv\mathfrak{f}\lambda^2\mod4\mathfrak{o}}c(\xi)q^{\xi/4}.
\]
Note that the summation does not depend on the choice of $\lambda\mod2\mathfrak{o}$.
Also since $\mathfrak{f}$ is odd, all $f_\lambda$ are distinct.
By the assumption for $f,$ we have
\[
f(z)=\sum_{\lambda\in\mathfrak{o}/2\mathfrak{o}}f_\lambda(4z).
\]
For any $x\in\mathfrak{f}\hat{\mathfrak{d}}^{-1}$ where $\hat{\mathfrak{d}}=\boldsymbol{\delta}\hat{\mathfrak{o}}$, by Lemma \ref{behavioursofusharpandm}, one easily see that
\begin{align*}
&\rho(\mathbf{u}^\sharp(x))f(z/4)\\
=&\sum_{\lambda\in\mathfrak{o/2\mathfrak{o}}}\psi_1(\mathfrak{f}x\lambda^2/4)f_\lambda(z)\\
=&\sum_{\lambda\in\mathfrak{o/2\mathfrak{o}}}\psi(x\lambda^2/4)f_\lambda(z).
\end{align*}
Since $x\mapsto\psi(x\lambda^2/4)$ form $2^n$ distinct characters in $x\in\mathfrak{f}\hat{\mathfrak{d}}^{-1}$ for $\lambda\in\mathfrak{o}/2\mathfrak{o},$ the above implies that $f_\lambda$ can be written as a linear combination of some $\rho(\mathbf{u}^\sharp(x))f(z/4)$'s, that is, 
\[
f_\lambda\in\mathcal{V}\quad(\lambda\in\mathfrak{o}/2\mathfrak{o}).
\]
Now for $\lambda\in\mathfrak{o}/2\mathfrak{o}$ and $x\in\mathfrak{f}\hat{\mathfrak{d}}^{-1},$ we have
\[
\rho(\mathbf{u}^\sharp(x))f_\lambda(z)=\psi_1(x\lambda^2/4)f_\lambda(z).
\]
Remind that if $\gamma\in\widetilde{\Gamma_0(4\mathfrak{I})},$ then
\[
\rho(\gamma)f(z)=\varepsilon(\gamma)^{-1}f(z)
\]
where $\varepsilon=\prod_{v<\infty}\varepsilon_v$ is the genuine character of $\widetilde{\Gamma_{\mathrm{f}}}$ and $\varepsilon_v$ is the local character of $\widetilde{\Gamma_v}$ given by $\omega_{\psi_v}(\gamma_v)\phi_{0,v}=\varepsilon_v(\gamma_v)^{-1}\phi_{0,v}$ as in Lemma \ref{defofvarepsilon}.
Now set $\Gamma[4\mathfrak{f}\mathfrak{d}^{-1},\mathfrak{I}\mathfrak{f}^{-1}\mathfrak{d}]_\mathrm{f}=\prod_{v<\infty}\Gamma[4\mathfrak{f}\mathfrak{d}^{-1},\mathfrak{I}\mathfrak{f}^{-1}\mathfrak{d}]_v=\prod_{v<\infty}\mathbf{m}(2)_v\Gamma_v\mathbf{m}(2)_v^{-1}$.
We let $\check{\varepsilon}=\prod_{v<\infty}\check{\varepsilon}_v$ be the genuine character of $\widetilde{\Gamma[4\mathfrak{d}^{-1},\mathfrak{d}\mathfrak{I}]_\mathrm{f}}$ where $\check{\varepsilon}_v$ is the local genuine character of $\widetilde{\Gamma[4\mathfrak{f}\mathfrak{d}^{-1},\mathfrak{I}\mathfrak{f}^{-1}\mathfrak{d}]_v}$ given by $\omega_{\psi_v}(\gamma'_v)\phi'_{0,v}=\varepsilon_v(\gamma'_v)^{-1}\phi'_{0,v}$ as in Lemma \ref{defofcheckvarepsilon}.
Then for $\gamma'\in\widetilde{\Gamma[4\mathfrak{f}\mathfrak{d}^{-1},\mathfrak{I}\mathfrak{f}^{-1}\mathfrak{d}]_\mathrm{f}},$ we have
\begin{align*}
&\rho(\gamma')f(z/4)\\
=&2^{\sum_{i=1}^n(k_i+1/2)}\rho(\gamma'\mathbf{m}(2_\mathrm{f}))f(z)\\
=&2^{\sum_{i=1}^n(k_i+1/2)}\varepsilon(\mathbf{m}(2_\mathrm{f})^{-1}\gamma'\mathbf{m}(2_\mathrm{f}))^{-1}\rho(\mathbf{m}(2_\mathrm{f}))f(z)\\
=&\check{\varepsilon}(\gamma')^{-1}f(z/4)
\end{align*}
by Lemma \ref{behavioursofusharpandm} and Lemma \ref{defofcheckvarepsilon}.
By Lemma 3.1 of \cite{Ren:16}, we obtain that $\sum_{\lambda\in\mathfrak{o}/2\mathfrak{o}}\mathbb{C}\cdot f_\lambda$ forms a invariant irreducible subspace of $\mathcal{V}$ equivalent to $\Omega_\psi$ under the intertwining map $f_\lambda\mapsto\phi_\lambda$.
(Although for the local case, Lemma 3.1 of \cite{Ren:16} only treats the condition that $\mathfrak{I}_v=\mathfrak{o}_v,$ here $\mathfrak{I}$ is odd and if $v$ is odd, that lemma tells us nothing since whose result is already included in the condition. Hence one can apply that lemma here without any problem.)
Hence
\[
\rho(e^K)f_0=f_0
\]
and
\begin{align*}
&\rho(\mathbf{w}_{2\boldsymbol{\delta}})f_0\\
=&\prod_{v<\infty}(2,\boldsymbol{\delta}_v)_v\times\rho(\mathbf{m}(2_f^{-1})\mathbf{w}_{\boldsymbol{\delta}})f_0\\
=&2^{-n/2}\prod_{v<\infty}\overline{\alpha_{\psi_v}(\boldsymbol{\delta}_v)}(2,\boldsymbol{\delta}_v)_v\times\rho(\mathbf{m}(2_f^{-1}))\sum_{\lambda\in\mathfrak{o}/2\mathfrak{o}}f_\lambda\\
=&2^{\sum_{i=1}^nk_i}\prod_{v<\infty}\overline{\alpha_{\psi_v}(\boldsymbol{\delta}_v)}(2,\boldsymbol{\delta}_v)_v\times f
\end{align*}
by the definition of the Weil representation and Lemma \ref{behavioursofusharpandm}.
Now by the definition of $E^K,$ we see that $\rho(E^K)f=f$.
\end{proof}

From Proposition \ref{plusimpliesfourier} and \ref{fourierimpliesplus}, we have the following theorem, which is an analogue of Theorem 13.5 in \cite{HiraIke:13}.

\begin{theorem}
The Kohnen plus spaces $M^+_{k+1/2}(\Gamma_0(4\mathfrak{I}),\chi_\mathfrak{f})$ and $S^+_{k+1/2}(\Gamma_0(4\mathfrak{I}),\chi_\mathfrak{f})$ are the subspaces of $M_{k+1/2}(\Gamma_0(4\mathfrak{I}),\chi_\mathfrak{f})$ and $S_{k+1/2}(\Gamma_0(4\mathfrak{I}),\chi_\mathfrak{f}),$ respectively, which consist of the forms whose $\xi$-th Fourier coefficient vanishes unless there exists some $\lambda\in\mathfrak{o}$ such that $\xi-\mathfrak{f}\lambda^2\in4\mathfrak{o}$.
\end{theorem}

Although we have used an alternative way to define the plus spaces, it turns out if we consider the cases $F=\mathbb{Q}$ or $\mathfrak{I}=\mathfrak{o}$ (in which case $\mathfrak{f}$ is a unit such that $N_{F/\mathbb{Q}}(\mathfrak{f})=(-1)^{\sum_{i}k_i}$), the definition of the plus spaces coincide with the ones given by Kohnen or Hiraga \& Ikeda in \cite{Kohnen:82} or \cite{HiraIke:13}, respectively.

%%%%%%%%%%%%%%%%%%%%%%%%%%%%%%%%%%%%%%%%%%%%%%%%%%%%%%%%%%%%%%%%%%%%%%%%%%%%%%%%%%%%%%%%%%%%%%%5%%%%%%%5%%%%%%%5%%%%%%%5%%%%%%%5%%%%%%%5%%%%%%%5%%%%%%%5%%%%%%%5%%%%%%%5%%%%%%%5%%%%%%%5%%%%%%%5%%%%%%%5%%%%%%%%5

\section{Steinberg Representations}\label{Steinberg Representations}
We use the same notations as in Section \ref{Weil representation}.
Assume that $F$ is non-archimedean and $q$ is odd.
For $s\in\mathbb{C},$ let $\tilde{I}_\psi(s)$ be the principal series representation of $\Mp_2(F)$ given in Section \ref{The Idempotents $e^K$ and $E^K$}.
It is known that $\tilde{I}_\psi(s)$ is irreducible if $q^{2s}\neq q^{\pm1}$.
If $q^{2s}=q,$ there exists a short exact sequence
\[
0\longrightarrow St_\psi \longrightarrow \tilde{I}_\psi(s) \longrightarrow \omega_\psi^+ \longrightarrow 0
\]
of representations of $\Mp_2(F)$.
Here $St_\psi$ is a subrepresentation of $\tilde{I}_\psi(s)$ called the Steinberg representation and $\omega_\psi^+$ is the subrepresentation of the Weil representation $\omega_\psi$ consisting of the even functions.
Both $St_\psi$ and $\omega_\psi^+$ are irreducible.
\par
Note that if $\psi'$ is another non-trivial additive character of $F,$ there exists some unit $\xi\in F^\times$ such that $\psi'(x)=\psi(\xi x)$.
The two Steinberg representations $St_\psi$ and $St_{\psi'}$ are equivalent if and only if $\xi\in{F^\times}^2$.
Since we have assumed that $q$ is odd, this gives us that there are exactly two distinct Steinberg representations with respect to a non-trivial additive character of $F$.
The following lemma tells that they can be described just in terms of the complex number $s$.
\begin{lemma}\label{Stbgreps}
If $\psi'(x)=\psi(\xi x)$ for some non-square non-zero $\xi,$ then $\tilde{I}_{\psi'}(s)=\tilde{I}_\psi(s')$ for some $s'$ such that $q^{s'}=-q^s$.
In particular, we have $St_{\psi'}\subset\tilde{I}_{\psi}(s')$.
\end{lemma}
\begin{proof}
Let $f\in\tilde{I}_{\psi'}(s)$ such that $f(g)\neq0$.
By symmetry it suffices to show that $f\in\tilde{I}_\psi(s'),$ that is, to show that
\[
f(\mathbf{m}(a)g)=\frac{\alpha_\psi(1)}{\alpha_\psi(a)}|a|^{s'+1}f(g)
\]
for any $a\in F^\times$.
Denote the order of $a$ by $\ord(a)$.
We have to prove
\[
\frac{\alpha_\psi(1)}{\alpha_\psi(a)}|a|^{s'+1}=\frac{\alpha_{\psi'}(1)}{\alpha_{\psi'}(a)}|a|^{s+1},
\]
or equivalently,
\[
\frac{\alpha_\psi(1)\alpha_{\psi}(\xi a)}{\alpha_\psi(a)\alpha_{\psi}(\xi)}=(-1)^{\ord(a)}.
\]
By the property of the Weil constants, the left side is equal to $(a,\xi)$.
Since $\xi$ is not a square, the equation now follows from the non-degeneracy of the Hilbert symbol $(\quad,\quad)$.
\end{proof}
\par
We denote the map from $\tilde{I}_\psi(s)$ to $\omega_\psi^+$ in the short sequence above by $\mathcal{S}$.
If we take the dual of this short sequence, we get
\[
0 \longrightarrow \omega_{\overline{\psi}}^+ \longrightarrow \tilde{I}_{\overline{\psi}}(-s) \longrightarrow St_{\overline{\psi}} \longrightarrow 0.
\]
The map from $\omega_{\overline{\psi}}^+$ to $\tilde{I}_{\overline{\psi}}(-s)$ is the dual $\mathcal{S}^\ast$ of $\mathcal{S}$.
It is given by $\left(\mathcal{S}^\ast(\phi)\right)(g)=\omega_{\overline{\psi}}^+(g)\phi(0)$.

\par

Let $\Gamma=\Gamma_0(\varpi)$ or $\Gamma=\Gamma[\varpi\mathfrak{d}^{-1},\mathfrak{d}]$.
Since $q$ is odd, we can construct a splitting of $\widetilde{\Gamma}$ over $\Gamma$ by the homomorphism
\[
\gamma\mapsto(\gamma,\varepsilon([\gamma]))\quad(\gamma\in\Gamma)
\]
where $\varepsilon$ is the restricted character on $\widetilde{\Gamma}$ of the one corresponding to $\psi$ given in Lemma \ref{defofvarepsilon}. 
We can consider $\Gamma$ as a subgroup of $\widetilde{\Gamma}$ in this way thus also a subgroup of $\Mp_2(F)$.
Now for any representation $\pi$ of $\Mp_2(F),$ we denote the $\Gamma$-fixed subspace of $\pi$ by $\pi^\Gamma$.
The functor $\pi\mapsto\pi^{\Gamma}$ is exact since $\Gamma$ is compact.
Thus if we apply it to the two sequence above, we get
\[
0\longrightarrow St_\psi^\Gamma \longrightarrow \tilde{I}_\psi(s)^\Gamma \longrightarrow {\omega_\psi^+}^\Gamma \longrightarrow 0
\]
and
\[
0 \longrightarrow {\omega_{\overline{\psi}}^+}^\Gamma \longrightarrow \tilde{I}_{\overline{\psi}}(-s)^\Gamma \longrightarrow St_{\overline{\psi}}^\Gamma \longrightarrow 0.
\]
By calculation, we can get that both ${\omega_\psi^+}^\Gamma$ and ${\omega_{\overline{\psi}}^+}^\Gamma$ are the one-dimensional space spanned by $\phi_0,$ the characteristic function of $\mathfrak{o}$.
As stated in Section \ref{The Hecke operators}, the space $\tilde{I}_\psi(s)^\Gamma$ is spanned by two unique functions $f_1$ and $f_2$.
Note that if we put $P=B\cap\Gamma_0(1)$ and set $\mathrm{Vol}(\widetilde{\Gamma})=1,$ then
\[
\Vol(\widetilde{P}\widetilde{\Gamma})=\begin{cases}
1&\mbox{ if }\Gamma=\Gamma_0(\varpi),\\
q&\mbox{ if }\Gamma=\Gamma[\varpi\boldsymbol{\delta}^{-1},\boldsymbol{\delta}].
\end{cases}
\]
and $\Vol(\widetilde{P}\mathbf{w}_{\boldsymbol{\delta}}\widetilde{\Gamma})=\Vol(\widetilde{P}\widetilde{\Gamma})^{-1}q$.
We can define $f'_1$ and $f'_2$ which span $\tilde{I}_{\overline{\psi}}(-s)^\Gamma$ in the similar manner with $f_1$ and $f_2$.
By taking the values on $I$ and $\mathbf{w}_{\boldsymbol{\delta}}$, one gets
\[
\mathcal{S}^\ast(\phi_0)=f'_1+f'_2.
\]
\par
Now if $f\in St_\psi^\Gamma\hookrightarrow\tilde{I}_\psi(s)^\Gamma,$ it is in the kernel of $\mathcal{S},$ hence
\begin{align*}
0&=\left(\mathcal{S}(f),\phi_0\right)\\
&=(f,\mathcal{S}^\ast(\phi_0))\\
&=(f,f'_1+f'_2),
\end{align*}
from which we get that
\[
f\in\mathbb{C}\cdot\left(f_1-q^{-1}f_2\right).
\]
We write these results into a lemma.
\begin{lemma}\label{lem:fixedsusbofSt}
We have
\[
St_\psi^\Gamma=\mathbb{C}\cdot(f_1-q^{-1}f_2).
\]
\end{lemma}
\par
Now assume $\Gamma=\Gamma_0(\omega)$.
Fix a unit $\eta\in\mathfrak{o}^\times$ and define the additive character $\psi_1$ of $F$ by $\psi(x)=\psi_1(\eta x)$.
For $f\in\tilde{I}_\psi(s)$ and $0\neq\xi\in\mathfrak{o},$ the Whittaker function of $f$ and $\xi$ is a function on $\Mp_2(F)$ given by
\[
\mathcal{W}_{f,\xi}(g)=\int_F f(\mathbf{w}_{\boldsymbol{\delta}}\mathbf{u}^\sharp(x)g)\overline{\psi_1(\xi x)}dx.
\]
Now set $f=f_1-q^{-1}f_2$.
We are interested in the value of $\mathcal{W}_{f,\xi}$ at $I$ for $\xi\in\mathfrak{o}^\times$.
In fact, because
\begin{align*}
&\mathcal{W}_{f_2,\xi}(I)\\
=&\int_F f_2(\mathbf{w}_{\boldsymbol{\delta}}\mathbf{u}^\sharp(x))\overline{\psi_1(\xi x)}dx\\
=&\int_{\mathfrak{d}^{-1}}f_2(\mathbf{w}_{\boldsymbol{\delta}}\mathbf{u}^\sharp(x))\overline{\psi_1(\xi x)}dx\\
=&q^{c_\psi}
\end{align*}
and by the calculations in Section 6 of \cite{HiraIke:13}
\[
\mathcal{W}_{f_1+f_2,\xi}(I)=q^{c_\psi}\left(1+\left(\frac{\xi\eta^{-1}}{\mathfrak{p}}\right)q^{-s-1/2}\right)
\]
where
\[
\left(\frac{\xi\eta^{-1}}{\mathfrak{p}}\right)=
\begin{cases}
1&\mbox{ if }\xi\eta^{-1}\in (\mathfrak{o}^\times)^2,\\
-1&\mbox{ if }F(\sqrt{\xi\eta^{-1}})/F\mbox{ is an unramified extension},
\end{cases}
\]
we have
\begin{equation}\label{whittakeratI}
\mathcal{W}_{f,\xi}(I)=q^{c_\psi-1}\left(q^{1/2-s}\left(\frac{\xi\eta^{-1}}{\mathfrak{p}}\right)-1\right).
\end{equation}
Now if $\tilde{I}_{\psi}(s)=\tilde{I}_{\psi_1}(s_1)$ for some $s_1\in\mathbb{C},$ Lemma \ref{Stbgreps} implies that we can also write
\begin{equation}\label{whittakeratI2}
\mathcal{W}_{f,\xi}(I)=q^{c_\psi-1}\left(q^{1/2-s_1}\left(\frac{\xi}{\mathfrak{p}}\right)-1\right).
\end{equation}

\section{New Forms}
In this section we retain the notations given in Section \ref{Automorphic forms on MpA}.
Note that the domain of the global character $\varepsilon$ can be enlarged to the congruence subgroup $\widetilde{\Gamma[\mathfrak{d}^{-1},4\mathfrak{d}]_\mathrm{f}}$.
We denote this character also by $\varepsilon$.
\par
Let $v$ be a finite place of $F$.
Given a representation $(\pi,\mathcal{V})$ of $\Mp_2(F_v),$ we say that $\pi$ is unramified if there exists a non-zero vector $\mathbf{v}\in\mathcal{V}$ such that
\[
\pi(\gamma)\mathbf{v}=\varepsilon_v(\gamma)^{-1}\mathbf{v}
\]
for $\gamma\in\widetilde{\Gamma[\mathfrak{d}_v^{-1},4\mathfrak{d}_v]_v}$.
Such $\mathbf{v}$ is called an unramified vector.
It is known that if $\pi$ is irreducible and unramified, then it is contained in a principal series $\tilde{I}_{\psi_v}(s_v)$ for some $s_v\in\mathbb{C}$.
\par
For a cusp form $f\in S^+_{k+1/2}(\Gamma_0(4\mathfrak{I}),\chi_\mathfrak{f}),$ the corresponding automorphic form $\Phi_f\in\mathcal{A}^{\tiny\mbox{CUSP}}_{k+1/2}(\SL_2(F)\backslash \Mp_2(\mathbb{A});\widetilde{\Gamma_0(4\mathfrak{I})_\mathrm{f}},\varepsilon)$ generates a representation of $\Mp_2(\mathbb{A}_\mathrm{f})$ via the right translation $\rho$.
We may simply call $\rho$ the representation corresponding to $f$.
If $\rho$ is irreducible, then it has the form $\rho=\otimes_{v<\infty}\rho_v$ where $\rho_v$ is a local irreducible representation of $\Mp_2(F_v)$.
In that case, we call $f$ a primitive form.
\begin{definition}
Using the notations above, the space of old forms $S^{+,\mathrm{OLD}}_{k+1/2}(\Gamma_0(4\mathfrak{I}),\chi_\mathfrak{f})$ is the subspace of $S^+_{k+1/2}(\Gamma_0(4\mathfrak{I}),\chi_\mathfrak{f})$ spanned by all cusp forms $f$ such that the representation corresponding to $f$ is unramified at some place $v$ dividing $\mathfrak{I}$.
The space of new forms $S^{+,\mathrm{NEW}}_{k+1/2}(\Gamma_0(4\mathfrak{I}),\chi_\mathfrak{f})$ is the orthogonal complement of $S^{+,\mathrm{OLD}}_{k+1/2}(\Gamma_0(4\mathfrak{I}),\chi_\mathfrak{f})$ in $S^+_{k+1/2}(\Gamma_0(4\mathfrak{I}),\chi_\mathfrak{f})$ under the Petersson inner product.
\end{definition}

\par

Recall that 
\[
\Gamma_v=\begin{cases}
\Gamma_0(1)_v&\mbox{ if }v\nmid2\mathfrak{I},\\
\Gamma_0(\varpi_v)_v&\mbox{ if }v\mid\mathfrak{f}^{-1}\mathfrak{I},\\
\Gamma[\varpi_v\mathfrak{d}_v^{-1},\mathfrak{d}_v]_v&\mbox{ if }v\mid\mathfrak{f},\\
\Gamma_0(4)_v&\mbox{ if }v\mid2.
\end{cases}
\]
Put $\widetilde{\mathcal{H}_v}=\widetilde{\mathcal{H}}(\widetilde{\Gamma_v}\backslash\Mp_2(F_v)/\widetilde{\Gamma_v};\varepsilon_v)$ for any finite place $v$.
Using Hecke operators we may give an alternative definition of the old space.
Given any integral ideal $\mathfrak{I}'$ such that $\mathfrak{I}'\mid\mathfrak{I},$ the space $S^+_{k+1/2}(\Gamma_0(4\mathfrak{I}'),\chi_\mathfrak{f})$ (which is the zero space if $\mathfrak{f}\nmid\mathfrak{I}'$) is clearly a subspace of $S^+_{k+1/2}(\Gamma_0(4\mathfrak{I}),\chi_\mathfrak{f})$.
Fix a finite place $v$ which divides $\mathfrak{I}$ but not $\mathfrak{I}'$.
Any primitive form in $S^+_{k+1/2}(\Gamma_0(4\mathfrak{I}'),\chi_\mathfrak{f})$ generates an unramified representation at $v$ and so does its image under any Hecke operator in $\widetilde{\mathcal{H}_v},$ which generates the same representation actually if it is non-zero.
Conversely, if a primitive form in $S^+_{k+1/2}(\Gamma_0(4\mathfrak{I}),\chi_\mathfrak{f})$ generates an unramified local representation at some $v\mid\mathfrak{I},$ then it lies in the image of $S^+_{k+1/2}(\Gamma_0(4\mathfrak{p}_v^{-1}\mathfrak{I}),\chi_\mathfrak{f})$ under some Hecke operator in $\widetilde{\mathcal{H}_v}$.
Thus the old space is indeed the subspace spanned by any cusp forms which are in the image of some $S^+_{k+1/2}(\Gamma_0(4\mathfrak{I}'),\chi_\mathfrak{f})$ under the Hecke operators.
Precisely, we have
\begin{align*}
 &S^{+,\mathrm{OLD}}_{k+1/2}(\Gamma_0(4\mathfrak{I}),\chi_\mathfrak{f})\\
=&\sum_{v\mid\mathfrak{I}}\left(S^+_{k+1/2}(\Gamma_0(4\mathfrak{p}_v^{-1}\mathfrak{I}),\chi_\mathfrak{f})+\rho_v(\widetilde{\mathcal{T}}_{1,v})S^+_{k+1/2}(\Gamma_0(4\mathfrak{p}_v^{-1}\mathfrak{I}),\chi_\mathfrak{f})\right).
\end{align*}
Here we only need to take the certain Hecke operator $\widetilde{\mathcal{T}}_{1,v}$ since the $\Gamma_v$-fixed subspace of a principal series $\tilde{I}_{\psi_v}(s_v)$ is two-dimensional.

\par

Note that the representation associating to any $f\in S^+_{k+1/2}(\Gamma_0(4\mathfrak{I}),\chi_\mathfrak{f})$ is unramified at all finite places $v$ not dividing $\mathfrak{I}$.
If $f$ is a primitive new form, then for all odd places $v$ which divide $\mathfrak{I}$ the local representations $\rho_v$ are equivalent to a Steinberg representation $St_{\psi'_v}$ introduced in Section \ref{Steinberg Representations} for some non-trivial additive character $\psi'_v$ of $F_v$.
Actually, we can find a basis $\mathcal{B}$ of $S^{+,\mathrm{NEW}}_{k+1/2}(\Gamma_0(4\mathfrak{I}),\chi_\mathfrak{f})$ consisting of such forms.
\par
Remind that the local character $\psi_{1,v}$ is the local part of the global character $\psi_1$ of $\mathbb{A}/F$ such that $\psi_{1,\infty}(x)=\mathbf{e}(x)$ for any infinite place $\infty$ of $F$.
We call the one of the two Steinberg representations which is not equivalent to $St_{\psi_{1,v}}\subset\tilde{I}_{\psi_{1,v}}(1/2)$ the twisted Steinberg representation.
Now by the arguments in Section \ref{Steinberg Representations} and (\ref{whittakeratI2}), we have the following proposition.

\begin{proposition}\label{prop:A+A-}
Fix an odd place $v$ which divides $\mathfrak{f}^{-1}\mathfrak{I}$.
Let $f$ be a new form in the basis $\mathcal{B}$ of $S^{+,\mathrm{NEW}}_{k+1/2}(\Gamma_0(4\mathfrak{I}),\chi_\mathfrak{f})$ with Fourier coefficients $c(\xi)$.
If the local representation $\rho_v$ corresponding to $f$ at $v$ is the non-twisted Steinberg representation, then $c(\xi)=0$ if $\left(\frac{\xi}{\mathfrak{p}_v}\right)=1$.
If $\rho_v$ is twisted, then $c(\xi)=0$ if $\left(\frac{\xi}{\mathfrak{p}_v}\right)=-1$.
\end{proposition}
We set the following two spaces
\[
A^+_v=\left\{f=\sum_{\xi}c(\xi)q^\xi\in S^{+,\mathrm{NEW}}_{k+1/2}(\Gamma_0(4\mathfrak{I}),\chi_\mathfrak{f})\,\bigg|\,c(\xi)=0\mbox{ if }\left(\frac{\xi}{\mathfrak{p}_v}\right)=1
\right\}
\]
and
\[
A^-_v=\left\{f=\sum_{\xi}c(\xi)q^\xi\in S^{+,\mathrm{NEW}}_{k+1/2}(\Gamma_0(4\mathfrak{I}),\chi_\mathfrak{f})\,\bigg|\,c(\xi)=0\mbox{ if }\left(\frac{\xi}{\mathfrak{p}_v}\right)=-1
\right\}.
\]
By the proposition, we have 
\begin{equation}\label{AplusandAminus}
S^{+,\mathrm{NEW}}_{k+1/2}(\Gamma_0(4\mathfrak{I}),\chi_\mathfrak{f})=A^+_v+A^-_v.
\end{equation}
\par
From now on we want to understand how the Hecke operators act on the new forms in $\mathcal{B}$.
Fix one $f\in\mathcal{B}$ with the representation $\rho=\prod_{v<\infty}\rho_v$ of $\Mp_2(\mathbb{A}_\mathrm{f})$.
\par
First, let $v$ be an odd place which does not divides $\mathfrak{I}$.
Since $f$ is an unramified vector in $\rho_v$ and the unramified subspace in $\rho_v$ is one-dimensional, we get that $f$ is an eigenvector for any Hecke operators in the Hecke algebra $\widetilde{\mathcal{H}_v}$.
In particular, we have the following proposition.
\begin{proposition}\label{prop:T1v1}
Assume that $v$ does not divide $2\mathfrak{I}$.
Let $\widetilde{\mathcal{T}}_{1,v}\in\widetilde{\mathcal{H}_v}$ be the Hecke operator defined in Section \ref{The Hecke algebras}.
We have
\[
\rho_v(\widetilde{\mathcal{T}}_{1,v})f=q_v^{1/2}(q_v^{s_v}+q_v^{-s_v})f
\] 
where $s_v$ is the complex number such that $\rho_v\subset\tilde{I}_{\psi_v}(s_v)$.
\end{proposition}
\begin{proof}
We know that $f\in\rho_v$ is equivalent to some non-zero unramified vector $f'$ in $\tilde{I}_{\psi_v}(s_v)$ up to multiplication of non-zero complex numbers under the actions of $\Mp_2(F_v),$ thus the proposition follows from Lemma \ref{lem:action of T1, gamma1}.
\end{proof}

\par

Secondly, let $v$ be an odd place which divides $\mathfrak{I}$. 
Now the local representation $\rho_\psi$ is a Steinberg representation $St_\psi$ contained in a principal series $\tilde{I}_{\psi_v}(s_v)$ for some $s_v\in\mathbb{C}$ such that $q_v^{2s_v}=q_v$.
In this case, as what we have seen in Section \ref{Steinberg Representations}, the subspace of $\rho_v$ fixed by $\Gamma_v$ is one-dimensional.
Since the operator $E^K_v$ projects any vector in $\rho_v$ into the $\Gamma_v$-fixed subspace, the form $f$ is also fixed by $\Gamma_v$.
Thus we get that $f$ is an eigenvector for any Hecke operators in the Hecke algebra $\widetilde{\mathcal{H}_v}$.
\par
The following corollary is comparable with Prop. 4 of \cite{Kohnen:82}.
Again, we use the notations given in Section \ref{The Hecke algebras} with a subscript $v$ for the Hecke operators.
\begin{corollary}\label{cor:ALinvo}
Suppose that $v$ divides $\mathfrak{I}$ but not $\mathfrak{f}$.
Let $\widetilde{\mathcal{U}}_{1,v}$ be the Atkin-Lehner involution defined in Section \ref{The Hecke algebras}.
The subspaces $A^+_v$ and $A^-_v$ given above are the $-\left(\frac{\mathfrak{f}_v}{\mathfrak{p}_v}\right)$-eigenspace and $\left(\frac{\mathfrak{f}_v}{\mathfrak{p}_v}\right)$-eigenspace of $\widetilde{\mathcal{U}}_{1,v},$ respectively.
\end{corollary}
\begin{proof}
Let $f\in\mathcal{B}$.
Since $f$ is a $\Gamma_0(\varpi_v)$-fixed vector in $\rho_v,$ which is a Steinberg representation $St_\psi$ contained in $\tilde{I}_{\psi_v}(s_v)$ for some $s_v\in\mathbb{C},$ by Lemma \ref{lem:fixedsusbofSt}, $f$ is equivalent to the unique function $f'$ in $St_\psi$ which is right invariant under $\widetilde{\Gamma_0(\varpi_v)}$ and satisfies
\[
f'(I_v)=1
\]
and
\[
f'(\mathbf{w}_{\boldsymbol{\delta}_v})=-q_v^{-1}.
\]
By Lemma \ref{lem:action of operators, gamma2}, we have
\[ 
\rho_v(\widetilde{\mathcal{U}}_{1,v})f'(I_v)=-q_v^{-1/2+s_v}
\]
and
\[
\rho_v(\widetilde{\mathcal{U}}_{1,v})f'(\mathbf{w}_{\boldsymbol{\delta}_v})=q_v^{-1/2-s_v}.
\]
Since $q_v^{s_v}\in\{\pm\sqrt{q}\},$ we have $q_v^{-s_v}=q_v^{s_v-1}$ and the above tells us that
\[
\rho_v(\widetilde{\mathcal{U}}_{1,v})f'=-q_v^{-1/2+s_v}f'.
\]
Now if $\rho_v$ is non-twisted, then one of 
\[
\begin{cases}
q_v^{s_v}=\sqrt{q_v}\\
\left(\frac{\mathfrak{f}_v}{\mathfrak{p}_v}\right)=1
\end{cases}
\]
and
\[
\begin{cases}
q_v^{s_v}=-\sqrt{q_v}\\
\left(\frac{\mathfrak{f}_v}{\mathfrak{p}_v}\right)=-1
\end{cases}
\]
is true.
Under both of the conditions we have $-q_v^{s_v-1/2}=-\left(\frac{\mathfrak{f}_v}{\mathfrak{p}_v}\right),$ that is, $f'$ is a $-\left(\frac{\mathfrak{f}_v}{\mathfrak{p}_v}\right)$-eigenfunction of $\widetilde{\mathcal{U}}_{1,v}$.
The case for $\rho_v$ twisted is similar.
\end{proof}
Inspired by a recent work \cite{BaPu:15} from Baruch and Purkait, we give an if-and-only-if condition in terms of Hecke operators for the determination of a new form.
Let $\widetilde{\mathcal{T}}_{1,v}$ be the local Hecke operator corresponding to $\mathbf{m}(\varpi_v)$ defined in Section \ref{The Hecke algebras}.
\begin{theorem}\label{thm:iffconditionfornewforms}
Let $f\in S^+_{k+1/2}(\Gamma_0(4\mathfrak{I}),\chi_\mathfrak{f})$.
Then $f$ is a new form if and only if
\[
\rho_v(\widetilde{\mathcal{T}}_{1,v}\widetilde{\mathcal{U}}_{1,v})f=-f=\rho_v(\widetilde{\mathcal{U}}_{1,v}\widetilde{\mathcal{T}}_{1,v})
\]
for any $v$ dividing $\mathfrak{I}$.
\end{theorem}
\begin{proof}
We may assume that $f$ is a primitive form.
If $f=0,$ then the theorem is trivial.
So we may assume that $f$ is non-zero.
We must show that for any $v$ dividing $\mathfrak{I}$ the local representation $\rho_v$ generated by $f$ is a Steinberg representation if and only if $\rho_v(\widetilde{\mathcal{T}}_{1,v}\widetilde{\mathcal{U}}_{1,v})f=-f=\rho_v(\widetilde{\mathcal{U}}_{1,v}\widetilde{\mathcal{T}}_{1,v})$.
Fix a place $v$ which divides $\mathfrak{I}$.
The representation $\rho_v$ is contained in a principal series $\tilde{I}_{\psi_v}(s_v)$ for some $s_v\in\mathbb{C}$.
In particular, $f$ is corresponding to some $f'\in\tilde{I}_{\psi_v}(s_v)$ which satisfies $\rho_v(\gamma)f'=\varepsilon_v(\gamma)^{-1}f'$ for $\gamma\in\widetilde{\Gamma_0(\varpi_v)_v}$.
So we can write $f'$ in the form
\[
f'=c_1f_1+c_2f_2\quad(c_1, c_2\in\mathbb{C})
\]
where $f_1$ and $f_2$ are the spanning vectors of the $\widetilde{\Gamma_v}$-fixed subspace of $\tilde{I}_{\psi_v}(s_v)$ given in Section \ref{The Hecke operators}.
\par
Now since $\widetilde{\mathcal{T}}_{1,v}=\widetilde{\mathcal{U}}_{0,v}\widetilde{\mathcal{U}}_{1,v},$ by Lemma \ref{lem:action of operators, gamma2}, we have
\[
\rho_v(\widetilde{\mathcal{T}}_{1,v}\widetilde{\mathcal{U}}_{1,v})f'=q_vc_2f_1+(c_1+(q_v-1)c_2)f_2
\]
and
\[
\rho_v(\widetilde{\mathcal{U}}_{1,v}\widetilde{\mathcal{T}}_{1,v})f'=((q_v-1)c_1+q^{1+2s_v}c_2)f_1+q_v^{-2s_v}c_1f_2.
\]
By comparing the coefficients, we get that $\rho_v(\widetilde{\mathcal{T}}_{1,v}\widetilde{\mathcal{U}}_{1,v})f'=-f'=\rho_v(\widetilde{\mathcal{U}}_{1,v}\widetilde{\mathcal{T}}_{1,v})f'$ if and only if

\[
-q_vc_2=c_1=-q_v^{2s_v}c_2.
\]

Since $f$ is non-zero, both $c_1$ and $c_2$ are non-zero.
This implies that $q_v^{2s_v}=q_v$ and $f'\in\mathbb{C}\cdot(f_1-q_v^{-1}f_2)$.
Thus the representation of $\Mp_2(F_v)$ generated by $f'$ is a Steinberg representation.
\end{proof}
Using a proof similar to the one of Proposition \ref{prop:T1v1}, one easily get the following.
\begin{proposition}\label{prop:T1v2}
	For $v\mid\mathfrak{I},$ the Hecke operator $\widetilde{\mathcal{T}}_{1,v}\in\widetilde{\mathcal{H}_v}$ acts on $f\in\mathcal{B}$ by
	\[
	\rho_v(\widetilde{\mathcal{T}}_{1,v})f=q_v^{s_v-1/2}f
	\]
	where $s_v$ is the complex number such that $\rho_v\subset\tilde{I}_{\psi_v}(s_v)$.
\end{proposition}

\par
Finally we want to consider the case $v\mid2$.
In this case we also want to get an analogue of Proposition \ref{prop:T1v1} and \ref{prop:T1v2}, but now the $E^K_v$ fixed subspace of $\rho_v$ is not invariant under the action of $\widetilde{\mathcal{T}}_{1,v},$ thus we want to consider the Hecke operator $E^K_v\ast\widetilde{\mathcal{T}}_{1,v},$ which obviously leaves the $E^K_v$-fixed subspace invariant by the idempotence of $E^K_v$.

\begin{lemma}\label{lem:valuesoff+}
Let $f^+_v\in\tilde{I}_{\psi_v}(s_v)$ be the one such that $\tilde{I}_{\psi_v}(s_v)^{E^K_v}=\mathbb{C}\cdot f^+_v$ as in Proposition \ref{prop:EKfixedsubsp}.
Then we have
\begin{enumerate}
\item $f^+_v(I_v)=\overline{\alpha_{\psi_v}(2\boldsymbol{\delta}_v)}|2|_v^{-(s_v+1/2)}$.\\
\item $f^+_v(\mathbf{w}_{\boldsymbol{\delta}_v})=\overline{\alpha_{\psi_v}(2\boldsymbol{\delta}_v)\alpha_{\psi_v}(\boldsymbol{\delta}_v)}|2|_v^{s_v+1}$.\\
\item For $c\in\mathfrak{o}_v$ such that $0<\ord_v(c)<2e_v,$ we have
\[ f^+_v(\mathbf{u}^\flat(\boldsymbol{\delta}_v))=\overline{\alpha_{\psi_v}(2\boldsymbol{\delta}_v)\alpha_{\psi_v}(\boldsymbol{\delta}_vc)}\left|\frac{2}{c}\right|_v^{s_v+1}\int_{\mathfrak{o}_v}\psi_v\left(\frac{y^2}{\boldsymbol{\delta}_vc}\right)dy.
\]
In particular, $f^+_v(\mathbf{u}^\flat(\boldsymbol{\delta}_v))=0$ if $0<\ord_v(c)<2e_v$ is odd.
\end{enumerate}
\end{lemma}
\begin{proof}
The lemma follows from the definition of $f^+_v$ and Lemma \ref{lem:valuesofeK}.
The disappearance of $f^+_v(\mathbf{u}^\flat(\boldsymbol{\delta}_v))$ for $c\in\mathfrak{o}_v$ such that $0<\ord(c)<2e_v$ is odd follows from Lemma 3.6 of \cite{HiraIke:13}.
\end{proof}

Let $B_v$ be the Borel subgroup of $\SL_2(F_v),$ which consists of all upper-triangular matrices.
Then we can take
\[
\mathfrak{R}=\left\{
\mathbf{w}_{\boldsymbol{\delta}_v}, I_v
\right\}
\cup
\left\{
\mathbf{u}^\flat(\boldsymbol{\delta}_vc)\,\bigg|\,
c\in\varpi^r\cdot\mathfrak{o}_v^\times/({\mathfrak{o}_v^\times}^2+\mathfrak{p}_v^{2e_v-r}), 1\leq r\leq 2e_v-1
\right\}
\]
as a complete system of representatives of $\widetilde{B_v}\backslash\Mp_2(F_v)/\widetilde{\Gamma_0(4)_v}$ by Iwasawa decomposition.
For any $g\in\mathfrak{R},$ we let $f^+_{v,g}\in\tilde{I}_{\psi_v}(s_v)$ be the one which is supported on $\widetilde{B}g\widetilde{\Gamma_0(4)_v}$ and satisfies
\[
f^+_{v,g}(g)=f^+_v(g).
\]

\begin{lemma}\label{lem:T1vf+v}
We have
\[
\rho_v(\widetilde{\mathcal{T}}_{1,v})f^+_v=q_v^{3/2+s_v}f^+_v+(q_v^{-s_v+1/2}-q_v^{s_v+3/2})f^+_{v,I_v}.
\]
\end{lemma}
\begin{proof}
By Lemma \ref{lem:decompofdoublecosetsveven}, we have
\[
\rho_v(\widetilde{\mathcal{T}}_{1,v})f^+_v=q_v^{-1/2}\frac{\alpha_{\psi_v}(\varpi_v)}{\alpha_{\psi_v}(1)}\sum_{\xi\in\mathfrak{o}_v/\mathfrak{p}_v^2}\rho_v\left(\mathbf{u}^\sharp(\boldsymbol{\delta}_v^{-1}\xi)\mathbf{m}(\varpi_v)\right)f^+_v.
\]
Applying Lemma \ref{lem:valuesofeK} and Lemma \ref{lem:valuesoff+}, one sees 
\begin{align*}
&\rho_v(\widetilde{\mathcal{T}}_{1,v})f^+_v(I_v)=q_v^{s_v+1/2}f^+_v(I_v),\\
&\rho_v(\widetilde{\mathcal{T}}_{1,v})f^+_v(w_{\boldsymbol{\delta}_v})=q_v^{-s_v+3/2}f^+_v(w_{\boldsymbol{\delta}_v})\\
\mbox{and}\quad&\rho_v(\widetilde{\mathcal{T}}_{1,v})f^+_v(\mathbf{u}^\flat(\boldsymbol{\delta}_vc))=q_v^{-s_v+3/2}f^+_v(\mathbf{u}^\flat(\boldsymbol{\delta}_vc))
\end{align*}
for $0<\ord_v(c)<2e_v$.
From these formulas and the definition of $f^+_{v,I_v},$ one gets the lemma.
\end{proof}

From the lemma above we can get how $E^K_v\ast\widetilde{\mathcal{T}}_{1,v}$ acts on $f\in\mathcal{B}$.

\begin{proposition}\label{prop:T1v3}
	For $v\mid2,$ the Hecke operator $E^K_v\ast\widetilde{\mathcal{T}}_{1,v}\in\widetilde{\mathcal{H}_v}$ acts on $f\in\mathcal{B}$ by
	\[
	\rho_v(E^K_v\ast\widetilde{\mathcal{T}}_{1,v})f=(1+q_v^{-1})^{-1}q_v^{1/2}(q_v^{s_v}+q_v^{-s_v})f
	\]
	where $s_v$ is the complex number such that $\rho_v\subset\tilde{I}_{\psi_v}(s_v)$.
\end{proposition}
\begin{proof}
It suffices to show that
\[
\rho_v(E^K_v\ast\widetilde{\mathcal{T}}_{1,v})f^+_v=(1+q_v^{-1})^{-1}q_v^{1/2}(q_v^{s_v}+q_v^{-s_v})f^+_v.
\]
Since from Proposition \ref{prop:EKfixedsubsp} we already know that $f^+_v$ is fixed by $E^K_v,$ by Lemma \ref{lem:T1vf+v}, we only need to determine $\rho_v(E^K_v)f^+_{v,I_v}$.
Also because $E^K_v$ is an idempotent, we have $\rho_v(E^K_v)f^+_{v,I_v}\in\mathbb{C}\cdot f^+_v$
Hence we only need to calculate $\rho_v(E^K_v)f^+_{v,I_v}(I_v)$ to get $\rho_v(E^K_v)f^+_{v,I_v}$.
Note that
\begin{align*}
 &\widetilde{\Gamma[4^{-1}\mathfrak{d}_v^{-1},4\mathfrak{d}_v]}/\widetilde{\Gamma_v}\\
=&\left\{\mathbf{u}^\sharp(4^{-1}\boldsymbol{\delta}_v^{-1}\xi)\,\mid\,\xi\in\mathfrak{o}_v/(4\mathfrak{o}_v)\right\}\cup\left\{\mathbf{u}^\sharp(4^{-1}\boldsymbol{\delta}_v^{-1}\xi)\mathbf{w}_{4\boldsymbol{\delta}_v}\,\mid\,\xi\in\mathfrak{p}_v/(4\mathfrak{o}_v)\right\}
\end{align*}
and only the representatives of the form $\mathbf{u}^\sharp(4^{-1}\boldsymbol{\delta}_v^{-1}\xi)$ are contained in the support of $f^+_{v,I_v}$.
Thus we have
\begin{align*}
 &\rho_v(E^K_v)f^+_{v,I_v}(I_v)\\
=&\int_{\Mp_2(F_v)}E^K_v(h)f^+_{v,I_v}(h)dh\\
=&\sum_{\xi\in\mathfrak{o}_v/(4\mathfrak{o}_v)}\int_{\widetilde{\Gamma_v}}E^K_v(\mathbf{u}^\sharp(4^{-1}\boldsymbol{\delta}_v^{-1}\xi)\gamma)f^+_{v,I_v}(\mathbf{u}^\sharp(4^{-1}\boldsymbol{\delta}_v^{-1}\xi)\gamma)d\gamma\\
=&\Vol(\widetilde{K})^{-1}|2|_v^{-1}\int_{\mathfrak{o}_v}\sum_{\xi\in\mathfrak{o}_v/(4\mathfrak{o}_v)}\overline{\psi_v\left(\frac{\xi y^2}{4\boldsymbol{\delta}_v}\right)}dy\times f^+_{v,I_v}(I_v)\\
=&\frac{q_v^{e_v}}{q_v^{2e_v}+q_v^{2e_v-1}}\int_{2\mathfrak{o}_v}q_v^{2e_v}dh\times f^+_{v,I_v}(I_v)\\
=&(1+q_v^{-1})^{-1}f^+_{v,I_v}(I_v).
\end{align*}
From this we see
\[
\rho_v(E^K_v)f^+_{v,I_v}=(1+q_v^{-1})^{-1}f^+_v
\]
and by Lemma \ref{lem:T1vf+v} we have proved what we want to show.
\end{proof}

%%%%%%%%%%%%%%%%%%%%%%%%%%%%%%%%%%%%%%%%%%%%%%%%%%%%%%%%%%%%%%%%%%%%%%%%%%%%%%%%%%%%%%%%%%%%%%%%%%%%%%%%%%%%%%%%%%%%%%%%%%%%%%%%%%%%%%%%%%%%%%%%%%%%%%%%%%%%%%%%%%%%%%%%%%%%%%%%%%%%%%%%%%%%%%%%%%%%%%%%%%%%%%%%%%%%%%%%%%%%%%%%%%%%%%%%%%%%%%%%%%%%%%%%%%%%%%%%%%%%%%%%%%%%%%%%%%%%%%%%%%%%%%%%%%%%%%%%%%%%%%%%%%%%%%%%%%%%%%%%%%%%%%%%%%%%%%%%%%%%%%%%%%%%%%%%%%%%%%%%%%%%%%%%%%%%%%%%%%%%%%%%%%%%%%%%%%%%%%%

\section{Application of Waldspurger's results}\label{Application of Waldspurger's results}

In the last section we want to consider the relation between the new forms in the plus space and certain modular forms of weight $2k$.
We want to apply Waldspurger's theory for Shimura correspondence.
One can consult \cite{Gan:11} for the knowlodge of Shimura correspondence of automorphic representations.

\par

Given a primitive form $f\in S_{k+1/2}^{+,\mathrm{NEW}}(\Gamma_0(4N),\chi_\mathfrak{f})$ with the corresponding genuine irreducible representation $\rho'=\otimes_v\rho'_v,$ given by the right translation, of $\Mp_2(\mathbb{A}),$ we know that $\rho'_v$ is isomorphic to

\begin{itemize}
	\item a discrete series representation of minimal weight $k_i+1/2$ if $v=\infty_i$ is archimedean.
	\item a principal series representation $\tilde{I}_{\psi_v}(s_v)$ for some $s_v\in\mathbb{C}$ if $v<\infty$ and $v\nmid\mathfrak{I}$.
	\item a Steinberg representation $St_{\psi_v}$ contained in $\tilde{I}_{\psi_v}(s_v)$ for some $s_v\in\mathbb{C}$ such that $q_v^{2s_v}=q_v$ if $v<\infty$ and $v\mid\mathfrak{I}$.
\end{itemize}

Now assume that there exists some $i$ such that $k_i>1,$ then it is known that $\rho'$ cannot be a space consisting of theta functions on $\Mp_2(\mathbb{A})$.
Since $f$ is non-zero, there exists some $a\in F^\times$ such that the $\psi_{a}=\psi(a\cdot)$-th Fourier coefficient of $\rho'$ does not vanish.
Thus by Proposition 6.1 in \cite{Gan:11}, the theta lift $\Theta_{\psi_a}(\rho')$ of $\rho'$ to $\PGL_2(\mathbb{A})$ with respect to $\psi_a$ is not trivial.
Because that $\Theta_{\psi_a}(\rho')\otimes\hat{\chi}_a,$ where $\chi_a$ denotes the quadratic Hecke character of $\mathbb{A}^\times$ with respect to $a,$ does not depend on $a\in F^\times$ whenever $\Theta_{\psi_a}(\rho')$ is not trivial, we may put $\Wald_{\psi}(\rho')=\Theta_{\psi_a}(\rho')\otimes\hat{\chi}_a,$ which is called the Waldspurger lift of $\rho'$.
The Waldspurger lift can be written as a tensor product $\Wald_{\psi}(\rho')=\otimes_v\Wald_{\psi_v}(\rho'_v)$.

\par

According to the table given in 2.17 in \cite{Gan:11}, the local Waldspurger's lift $\Wald_{\psi_v}(\rho_v)$ is a representation of $\PGL_2(F_v)$ which is isomorphic to

\begin{itemize}
	\item a discrete series representation of minimal weight $2k_i$ if $v=\infty_i$ is archimedean.
	\item a principal series representation $I_v(s_v)$ which is induced from the character given by $\begin{pmatrix}
	a&b\\0&d
	\end{pmatrix}\mapsto|ad^{-1}|^{s_v}$ if $v<\infty$ and $v\nmid\mathfrak{I}$.
	\item a Steinberg representation $St_v,$ which is the only irreducible component contained in $I_v(s_v)$ if $v<\infty$ and $v\mid\mathfrak{I}$.
\end{itemize}

Waldspurger's theorems imply that in such a lift we get exactly all the irreducible representations of $\PGL_2(\mathbb{A})$ which has local components satisfying the above conditions.
Also, if we put $K_0(\mathfrak{I}_v)\subset \PGL_2(F_v)$ to be the standard maximal compact subgroup if $\infty>v\nmid\mathfrak{I}$ or the Iwahori subgroup if $v\mid\mathfrak{I}$, then we know that any Hecke automorphic form in $\mathcal{A}^{\mathrm{CUSP}}_{2k}(\PGL_2(F)\backslash\PGL_2(\mathbb{A})/\prod_{v<\infty}K_0(\mathfrak{I}_v))$ which generates a Steinberg representation at any $v\mid\mathfrak{I}$ must generate a global representation satisfying the same conditions.
We call such an automorphic form in $\mathcal{A}^{\mathrm{CUSP}}_{2k}(\PGL_2(F)\backslash\PGL_2(\mathbb{A})/\prod_{v<\infty}K_0(\mathfrak{I}_v))$ a new form.
Thus this gives us a one-to-one correspondence between the primitive forms in $S_{k+1/2}^{+,\mathrm{NEW}}(\Gamma_0(4N),\chi)$ and the Hecke new forms in $\mathcal{A}^{\mathrm{CUSP}}_{2k}(\PGL_2(F)\backslash\PGL_2(\mathbb{A})/\prod_{v<\infty}K_0(\mathfrak{I}_v))$.
Moreover, with the notations used in Section \ref{The Hecke algebras}, we let $\widetilde{\mathcal{U}}_{m,v}$ and $\widetilde{\mathcal{T}}_{m,v}$ correspond to the Hecke operators $\mathcal{U}_{m,v}$ and $\mathcal{T}_{m,v}$ of $\PGL_2(F_v),$ respectively, for any finite odd place $v$.
By the results about the eigenvalues of the Hecke operators in $\widetilde{\mathcal{H}}_{v}$ for finite odd places and the knowledge about the integral weight automorphic forms, we get that such a correspondence commutes with the actions of the local Hecke operators.
Also, for finite even $v,$ if we consider $(1+q_v^{-1})E^K_v\ast\widetilde{\mathcal{T}}_{1,v}$ and $\mathcal{T}_{1,v}$ which is defined similarly, we still get the similar result.
We write this in a theorem to finish this section.

\begin{theorem}
There exists a Hecke isomorphism between $S_{k+1/2}^{+,\mathrm{NEW}}(\Gamma_0(4N),\chi_\mathfrak{f})$ and the space $\mathcal{A}^{\mathrm{NEW,CUSP}}_{2k}(\mathfrak{I}_v)$
which is spanned by the Hecke forms in $\mathcal{A}^{\mathrm{CUSP}}_{2k}(\PGL_2(F)\backslash\PGL_2(\mathbb{A})/\prod_{v<\infty}K_0(\mathfrak{I}_v))$ which generates a Steinberg representation at any place dividing $\mathfrak{I}$.
\end{theorem}

\end{document}